\documentclass{article}
\usepackage{amssymb,amsmath,amsthm}
\usepackage{mathrsfs}
\usepackage{color}
\usepackage{authblk}
\usepackage{enumerate}
\usepackage[colorlinks=true,citecolor=blue,linkcolor=blue]{hyperref}

\newcommand{\Rl}{\mathbb{R}}
\newcommand{\Cplx}{\mathbb{C}}

\newcommand{\Ntrl}{\mathbb{N}}

\newcommand{\Ac}{\mathcal{A}}
\newcommand{\Bc}{\mathcal{B}}
\newcommand{\Kc}{\mathcal{K}}
\newcommand{\Lc}{\mathcal{L}}

\newcommand{\Dc}{\mathcal{D}}

\newcommand{\Zc}{\mathcal{Z}}

\newcommand{\tr}{\mathrm{tr}}

\newcommand{\Tr}{\mathrm{Tr}}

\newcommand{\sgn}{\operatorname{sgn}}
\newcommand{\diag}{\mathrm{diag}}

\newcommand{\dom}{\mathrm{dom}}
\newcommand{\loc}{\mathrm{loc}}

\newcommand{\rank}{\mathrm{rank}}

\newcommand{\Ch}{\mathrm{Ch}}

\newcommand{\bx}{\mathbf{x}}

\newcommand{\bbF}{\mathbf{F}}

\newcommand{\Dd}{\mathscr{D}}

\def\XXint#1#2#3{{\setbox0=\hbox{$#1{#2#3}{\int}$ }
\vcenter{\hbox{$#2#3$ }}\kern-.6\wd0}}

\numberwithin{equation}{section}

\newcommand{\pre}{pre}
\newcommand{\double}{clone~}
\newcommand{\doubling}{cloning~} 
\newcommand{\doublee}{clone} 
\newcommand{\doublinge}{cloning}

\newtheorem{theorem}{Theorem}[section]
\newtheorem{proposition}[theorem]{Proposition}
\newtheorem{corollary}[theorem]{Corollary}
\newtheorem{definition}[theorem]{Definition}
\newtheorem{lemma}[theorem]{Lemma}
\newtheorem{remark}[theorem]{Remark}

\newtheorem{hypothesis}[theorem]{Hypothesis}

\newcommand{\highlight}{}

\date{\today}

\begin{document}

    \title{Noncommutative Geometry for Symmetric Non-Self-Adjoint Operators}
    \author[1]{Alain Connes}
    \author[2]{Galina Levitina}
    \author[2]{Edward McDonald}
    \author[2]{Fedor Sukochev}
    \author[2]{Dmitriy Zanin}

    \affil[1]{{Coll\`ege de France, IHES, 3, rue d'Ulm, 75231 Paris cedex 05, France}}
    \affil[2]{{School of Mathematics and Statistics, University of New South Wales, Sydney NSW 2052, Australia }}
    
\maketitle{}
    \begin{abstract}
        We introduce the notion of a pre-spectral triple, which is a generalisation of a spectral triple $(\Ac, H, D)$ where $D$ is no longer required to be self-adjoint, but closed and symmetric. Despite having weaker assumptions, pre-spectral triples allow us to introduce noncompact noncommutative geometry with boundary. In particular, we derive the Hochschild character theorem in this setting. We give a detailed study of Dirac operators with Dirichlet boundary conditions on domains in $\Rl^d$, $d \geq 2$.
    \end{abstract}


\section{Introduction}

    The noncommutative geometric, or spectral, perspective on geometry is that a geometric space can be modelled by a spectral triple $(\Ac,H,D)$. A standard example
    of a spectral triple arises from a compact Riemannian spin manifold $(X,g)$ with a Dirac-type operator $D$ on sections of the spinor bundle. In this example we take $\Ac$ 
    to be the algebra of smooth functions on $X$ and $H$ to be the Hilbert space of square-integrable sections of the spinor bundle (see \cite[Chapter 11]{GVF} for further details of this class of examples). 
    A requirement for a spectral triple is that
    $D$ be a self-adjoint operator on $H$, and this is indeed the case for the Dirac-type operator just described.
    
    There has been recent work on developing a theory of noncommutative geometry for operators which are symmetric, but not necessarily self-adjoint. Such operators
    naturally arise when considering Dirac operators on domains $\Omega \subset \Rl^d$ with boundary conditions. 
    One approach to incorporating symmetric operators in noncommutative geometry was suggested by Blackadar \cite[Page 164]{Blackadar-KT}, although
    flaws with this approach were noted by later authors \cite{Forsyth-Mesland-Rennie-2014,Hilsum-2010}. In 2005, Bettaieb, Matthey and Valette \cite{Bettaieb-Matthey-Valette-2005}
    studied a similar problem in the abstract setting, however a number of technical difficulties in their work were explained by Forsyth, Mesland and Rennie in 2014 \cite{Forsyth-Mesland-Rennie-2014}. 
    A 2012 memoir from Lesch, Moscovici and Pflaum \cite{LMP-memoir} directly addressed the question of defining the Connes-Chern character for manifolds with boundary, using the framework of relative $K$-homology.
    However their work differs in methods and aims from the present text as we instead focus on non-self-adjoint operators.
    
    More recently Forsyth, Goffeng, Mesland and Rennie studied non-self-adjoint operators from the noncommutative geometric perspective \cite{Forsyth-Goffeng-Mesland-Rennie-arxiv}. A related topic is the unbounded perspective on $KK$-theory,
    and symmetric non-self-adjoint operators feature in this theory, see e.g. \cite{Mesland-Rennie-jfa-2016}. 
    A closely related subject matter is the description of manifolds with boundary in noncommutative geometry: in this direction we mention in particular the work of Schrohe \cite{Schrohe-1999}
    and more recently Iochum and Levy \cite{Iochum-Levy-boundary-2011} studied spectral triples associated to non-self-adjoint operators with a view to describing Dirac operators with boundary conditions.
    In fact as early as 1989, Baum, Douglas and Taylor \cite{Baum-Douglas-Taylor-1989} studied manifolds with boundary from the $K$-homological perspective. 
    Recently, van den Dungen \cite{van-den-Dungen-arxiv} studied perturbations of self-adjoint operators by certain symmetric operators. 
    
    A substantial impetus and the starting point for the present text is the work of Hilsum \cite{Hilsum-2010}. Hilsum studied in detail
    how a symmetric non-self-adjoint operator can define a cycle in $KK$-theory, and his work inspired our own definition of a \pre-spectral triple (Definition \ref{\pre-spectral triple definition}). 
    {\highlight Another related paper which furthers Hilsum's work, in particular in the $KK$-theoretic setting, is due to Kaad and van Suijlekom \cite{Kaad-Suijlekom}.}

\subsection{Motivation}
    We may illustrate the subtleties involved with non-self-adjoint operators with the following simple and well-known example: consider the open unit interval $(0,a)$, where $0 \leq a \leq \infty$.
    The Dirac operator with Dirichlet boundary conditions on the Hilbert space $L_2(0,a)$ is (formally) defined as the differentiation operator $Du = -iu'$, for $u$ a smooth
    compactly supported function on $(0,a)$. The linear operator $D$ may be given the domain $\dom(D) = W^{1,2}_0(0,a)$: the space of Sobolev $W^{1,2}$ functions vanishing on the endpoints (see Subsection \ref{sobolev spaces section} for a complete definition).
%
    With this domain, $D$ is closed and symmetric, but \emph{not} self-adjoint. Indeed, by definition $\dom(D^*)$ is the set of all $u \in L_2(0,a)$ such that there
    exists $w \in L_2(0,a)$ such that for all $v \in W^{1,2}_0(0,a)$ we have that $\langle Dv,u\rangle_{L_2(0,a)} = \langle v,w\rangle_{L_2(0,a)}$. An argument involving integration by parts
    shows that:
    \begin{equation*}
        \dom(D^*) = W^{1,2}(0,a),
    \end{equation*}
    the space of Sobolev $W^{1,2}$ functions with no restrictions on the values at the endpoints. Moreover, $D^*u = -iu'$ for $u\in W^{1,2}(0,a)$. So here $D$ is symmetric but not self-adjoint,
    as $D^*$ properly extends $D$. 
    
    
    A consequence of $D$ being not self-adjoint is that the (self-adjoint) operators $D^*D$ and $DD^*$ are completely distinct self-adjoint extensions of $D^2$. Indeed,
    \begin{align*}
        \dom(D^*D) &= \{u \in W^{1,2}_0(0,a)\;:\;u' \in W^{1,2}(0,a)\}\\
        \dom(DD^*) &= \{u \in W^{1,2}(0,a)\;:\;u' \in W^{1,2}_0(0,a)\}
    \end{align*}
    while $D^*Du = -u''$ and $DD^*u = -u''$ for $u$ in the respective domains of $D^*D$ and $DD^*$. Hence $D^*D$ and $DD^*$ are Laplace
    operators with different boundary conditions: $\dom(D^*D)$ consists of functions which vanish at the endpoints (this is the Dirichlet Laplace operator)
    and $\dom(DD^*)$ consists of functions whose derivative vanishes at the endpoints (this is the Neumann Laplace operator). 
    These operators are of quite different character: while $D^*D$ always has trivial kernel, $DD^*$ has $1$-dimensional kernel consisting of constant functions if $a < \infty$. In higher dimensions and for domains
    more complicated than an interval, the situation can be far more subtle due to the diversity of possible boundary conditions (see, for example, \cite{Schmidt-1995}). For details on higher dimensional examples, see Section \ref{model examples section}.
    
    Moreover, when $a<\infty$, $D$ has self-adjoint extensions. However for $(0,\infty)$ there are no self-adjoint extensions of $D$. For further discussion of this example, see \cite[Section 1.3.1]{Schmudgen-2012} and \cite[Chapter IV, Section 49]{Akhiezer-Glazman-1993}.
        
    Our interest in the setting of non-self-adjoint operators is to generalise the Connes Character formula, also known as the Hochschild character formula or theorem. Connes' Character formula (originating in \cite[Section 2.$\gamma$]{NCG-book})
    provides a means of computing the Hochschild class of the Chern character on $K$-homology in ``local" terms. To be precise, for a {\highlight even smooth $p$-dimensional spectral triple $(\Ac,H,D)$ with grading $\Gamma$} and a Fredholm
    module $(H,F)$ representing the class of $(\Ac,H,D)$ in $K$-homology\footnote{In the case that $D$ has trivial kernel, we may take $F = \sgn(D)$}, we have an equality of the following two Hochschild cocycles:
    \begin{align*}
        a_0\otimes a_1\otimes\cdots\otimes a_p &\mapsto \frac{1}{2}\tr(\Gamma F \prod_{k=0}^p [F,a_k])\\
        a_0\otimes a_1\otimes\cdots\otimes a_p &\mapsto \tr_\omega(\Gamma a_0\prod_{k=1}^p [D,a_k](1+D^2)^{-p/2}).
    \end{align*}
    for $a_0,a_1,\ldots,a_p \in \Ac$, when evaluated on a Hochschild cycle in $\Ac^{\otimes (p+1)}$, and where $\tr_\omega$ denotes a Dixmier trace (or more generally
    any normalised trace on $\Lc_{1,\infty}$). {\highlight If the spectral triple is odd, then the same identity holds with $\Gamma=1$}. Expositions and proofs of the Character formula in various settings may be found in \cite[Section 2.$\gamma$]{NCG-book}, \cite{Connes-original-spectral-1995}, \cite[Section 10.4]{GVF}, \cite{CPRS}, \cite{CRSZ}, \cite[Appendix C]{higson} and \cite[Section 4]{BF}.
    A recent proof which forms the basis of the present text and with emphasis on the case where $\Ac$ is nonunital is contained in \cite{SZ-asterisque}.
    
    If we attempt to generalise Connes' character formula to the non-self-adjoint setting, a number of obstacles present themselves:
    \begin{enumerate}
        \item{} A suitable replacement for a (smooth, $p$-dimensional) spectral triple is required for settings where $D$ is not self-adjoint.
        \item{} It is not clear how to define $F$ when $D$ is not self-adjoint.
        \item{} The operator $D^2$ should be given a suitable replacement when $D$ is not self-adjoint. As the above discussion for the interval $(0,a)$ shows, there is a substantial difference
                between the two positive self-adjoint operators $D^*D$ and $DD^*$.
    \end{enumerate}
    
    In this paper we propose solutions to these obstacles, and state a version of the Connes character formula which is valid in a non-self-adjoint setting (Theorem \ref{character theorem}).
    Our choice of definitions is motivated by the following model example: consider $\Omega$ an open bounded subset of $\Rl^d$ with smooth boundary,
    and take $D$ to be the Dirac operator with Dirichlet boundary conditions on $\Omega$ (we defer the precise description of this example until Section \ref{model examples section}).
    
    To replace a spectral triple we introduce the notion of a \pre-spectral triple (Definition \ref{\pre-spectral triple definition}). A \pre-spectral triple is defined analogously to a spectral triple, but without the requirement
    that $D$ be self-adjoint. Our definition is directly inspired by prior work of Hilsum \cite{Hilsum-2010}, and constructions along these lines can also be found in \cite{Bettaieb-Matthey-Valette-2005}, the approach suggested in 
    \cite[Page 164]{Blackadar-KT} and in \cite{Forsyth-Goffeng-Mesland-Rennie-arxiv}. 
    
    Our key new tool is that a \pre-spectral triple $(\Ac,H,D)$ can be ``tamed" into a genuine spectral triple by a procedure we call \doublinge. The \doubling procedure is based on the fact that if $D$ is closed and symmetric then the operator
    \begin{equation*}
        \begin{pmatrix} 0 & D^*\\ D & 0\end{pmatrix}
    \end{equation*}
    is self-adjoint on the domain $\dom(D)\oplus\dom(D^*)$ and that the following representation
    \begin{equation*}
        \Ac\ni a\mapsto  \frac 12 \begin{pmatrix} a & a\\ a & a\end{pmatrix}
    \end{equation*}
    of the algebra $\Ac$ sees this ``double" of $D$ as a single copy, thus in effect halving the double. This cloning procedure is moreover idempotent, when considering spectral triples up to those for which the action of the algebra is null. {\highlight Indeed, our ``cloned'' Hilbert space splits into an orthogonal direct sum, and in the self-adjoint case our algebra will act trivially on one component (see the argument at the end of Subsection \ref{spectral triple basic definitions} for further details).}
 
    We are then able to state the Connes character formula for a \pre-spectral triple in terms of its \doublee.
        
\subsection{Plan of this paper}
     In the following section, we introduce background material concerning operator inequalities.
     
     Afterwards, in Section \ref{pre spectral triples section} we develop the notion of a \pre-spectral triple and the \doubling construction. There, we define $p$-dimensional \pre-spectral triples (Definition \ref{dimension definition})
     and smoothly $p$-dimensional \pre-spectral triples (Definition \ref{def_smoothly_p_dim}).
     In Section \ref{sufficient conditions section} we return to the setting of spectral triples and describe sufficient conditions to state the character formula, and the main result of that section is 
     Corollary \ref{sufficient conditions for double to satisfy 1.2.1}, where we state sufficient conditions on a \pre-spectral triple so that we can state the character formula for its \doublee.
     
     Then in Section \ref{character formula section}, we state a version of the Connes Character Formula which is valid for \pre-spectral triple in Theorem \ref{character theorem}, which is a direct corollary of the Character Theorem
     for \emph{bona fide} spectral triples, in the specific form obtained in \cite{SZ-asterisque}.
     
     In Section \ref{model examples section} we discuss our model example of a symmetric operator to serve as motivation, and the remainder
     of the paper is dedicated to showing that our model examples do indeed satisfy the requirements of a \pre-spectral triple and the requirements for our version of the Character Formula to hold.
     
\section{Operators, ideals and traces}
    The following material is standard, and for further details we refer the reader to \cite{LSZ, Simon-trace-ideals-2005, Gohberg-Krein}.
    Let $H$ be a separable complex Hilbert space, and denote $\Bc(H)$  the algebra of bounded linear operators on $H$, and denote $\Kc(H)$ the ideal in $\Bc(H)$
    of compact operators.
    The sequence of singular values $\mu(T) = \{\mu(n,T)\}_{n=0}^\infty$ of a compact operator $T$ on $H$ is defined by
    \begin{equation*}
        \mu(n,T) := \inf\{\|T-R\|\;:\;\rank(R)\leq n\}.
    \end{equation*}
    Equivalently, $\mu(n,T)$ is the $n$th eigenvalue of $|T|$ listed in non-increasing order with multiplicities.
    
    For $p \in (0,\infty)$, the operator ideal $\Lc_p$ is the set of operators $T$ with $\mu(T)$ being $p$-summable, and $\Lc_{p,\infty}$ is the set of operators $T$ with $\mu(n,T) = O((n+1)^{-1/p})$,
    with corresponding quasi-norms:
    \begin{align*}
               \|T\|_{p} &:= \left(\sum_{n=0}^\infty \mu(n,T)^p\right)^{1/p},\\
        \|T\|_{p,\infty} &:= \sup_{n\geq 0} (n+1)^{1/p}\mu(n,T).
    \end{align*}
    A useful inequality is that for $p \geq 1$ there is a constant $c_p$ such that:
    \begin{equation}\label{interpolation inequality}
        \|T\|_1 \leq c_p\|T\|_{\frac{p}{p+1},\infty}^{\frac{p}{p+1}}\|T\|_{\infty}^{\frac{1}{p+1}}.
    \end{equation}
    
    For $q \in [1,\infty)$, we also consider the ideal $\Lc_{q,1}$, defined as the set of bounded operators $T$ on $H$ satisfying:
    \begin{equation*}
        \|T\|_{\Lc_{q,1}} := \sum_{n\geq 0} \frac{\mu(n,T)}{(n+1)^{1-\frac{1}{q}}} < \infty.
    \end{equation*}
    We have the following H\"older-type inequality, if $\frac{1}{p}+\frac{1}{q} = 1$ then:
    \begin{equation*}
        \|TS\|_1 \leq \|T\|_{p,\infty}\|S\|_{q,1}.
    \end{equation*}
    
    Given two bounded operators $T$ and $S$ on $H$, we say that $T$ is logarithmically submajorised by $S$, written $T\prec\prec_{\log} S$ if
    for all $n\geq 0$ we have:
    \begin{equation*}
        \prod_{k=0}^n \mu(k,T) \leq \prod_{k=0}^n \mu(k,S).
    \end{equation*}
    An important fact is that the $\Lc_{1,\infty}$ quasi-norm is monotone under logarithmic submajorisation. That is: if $T \prec\prec_{\log} S$, then \cite[Subsection 2.1.2]{SZ-asterisque}:
    \begin{equation}\label{monotonicity}
        \|T\|_{1,\infty} \leq e\|S\|_{1,\infty}.
    \end{equation}
    We also have that if $T$ and $S$ are positive and $p> 0$, then if $T \prec\prec_{\log} S$ we have:
    \begin{equation}\label{convexity}
        T^p\prec\prec_{\log} S^p.
    \end{equation}
    
    A fundamental result concerning logarithmic submajorisation is the Araki-Lieb-Thirring inequality \cite[Beginning of page 169]{Araki-1990} (see also \cite[Theorem 2]{Kosaki-alt-1992}), which states that $A$ and $B$ are positive bounded operators on $H$ and if $r \geq 1$ then:
    \begin{equation}\label{ALT inequality}
        |AB|^r \prec\prec_{\log} A^rB^r.
    \end{equation}
    So therefore if $A^rB^r \in \Lc_{1,\infty}$ then $AB \in \Lc_{r,\infty}$ and:
    \begin{equation*}
        \|AB\|_{r,\infty}^r \leq e\|A^rB^r\|_{1,\infty}.
    \end{equation*}
    
    A trace on $\Lc_{1,\infty}$ is a linear functional $\varphi:\Lc_{1,\infty}\to \Cplx$ which is unitarily invariant: i.e. for all unitary operators $U$ on $H$, and $T\in \Lc_{1,\infty}$ we have $\varphi(UTU^*) = \varphi(T)$. There
    are a plethora of traces on $\Lc_{1,\infty}$, including the well-known Dixmier traces.
    An important fact is that all traces on $\Lc_{1,\infty}$ vanish on $\Lc_1$, further details on this theory may be found in \cite[Section 5.7]{LSZ}. Finally, we call a trace normalised if $\varphi\left(\diag\left\{\frac{1}{n+1}\right\}_{n=0}^\infty\right) =  1$.

\section{Pre-spectral triples and \doubling}\label{pre spectral triples section}
\subsection{Definition of \pre-spectral triple and the \doubling construction}\label{spectral triple basic definitions}
    A \pre-spectral triple is a generalisation of the usual definition of a spectral triple (as in \cite{Connes-original-spectral-1995}, \cite[Definition 9.16]{GVF} and \cite[Definition 2.1]{CGRS2}), where the operator $D$ is no longer required to be self-adjoint but merely closed and symmetric.
    Prior work in defining analogues of spectral triples for non-self-adjoint operators includes that of Hilsum \cite{Hilsum-2010}, who developed the essential ideas contained in the following definition.
    \begin{definition}\label{\pre-spectral triple definition}
        A \pre-spectral triple is a triple $(\Ac,H,D)$ where
        \begin{enumerate}[{\rm (i)}]
            \item{}\label{hilbert space} $H$ is a Hilbert space
            \item{}\label{* algebra} $\Ac$ is a $*$-subalgebra of the algebra of bounded operators on $H$
            \item{}\label{closed symmetric} $D:\dom(D)\to H$ is a closed symmetric operator. 
            \item{}\label{partials assumption} For all $a \in \Ac$, we have that $a(\dom(D^*))\subseteq \dom(D)$, and the commutator $[D^*,a]:\dom(D^*)\to H$ has bounded extension, which we denote $\partial(a)$.
            \item{}\label{compact resolvent condition} {\highlight For all $a \in \Ac$ the operator $a(1+D^*D)^{-1/2}$ is compact.} 
        \end{enumerate}
        The Hilbert space $H$ may also be equipped with a grading $\gamma$, such that $\gamma$ commutes with $\Ac$ and anticommutes with $D$. In this case $(\Ac,H,D)$ is called an even \pre-spectral triple.
        
        If $D = D^*$, then $(\Ac,H,D)$ is called a spectral triple.
    \end{definition}
    
    The definition of a \pre-spectral triple should be compared with the conventional definition of a spectral triple \cite{Connes-original-spectral-1995}, we
    have relaxed the condition that $D$ be self-adjoint.
    It is also worth noting that if $\Ac$ is unital and its unit acts
as the identity in $H$ then the identity maps $\dom(D^*)$ into $\dom(D)$, so $\dom(D^*)\subseteq \dom(D)$; thus in this unital case $D$ is necessarily
    self-adjoint and we simply have a spectral triple. 
    
    \begin{remark}\label{pre_spectral_triple_remark}
        \begin{enumerate}[{\rm (i)}]
            \item{} {\highlight 
			 The definition of a pre-spectral triple should be compared to that of a half-closed chain, initiated in \cite{Hilsum-2010} and further studied in \cite[Definition 1]{Kaad-Suijlekom}.
 			The only difference is that a half-closed chain in the sense of \cite{Kaad-Suijlekom} is defined relative to a general Hilbert $C^*$-module, whereas a pre-spectral triple is defined   relative to a Hilbert space. While the cloning process which we introduce in Definition \ref{clone_def} will likely extend to half-closed chains,
			we have preferred to focus on the Hilbert space case. In this more special case, we have access not only to the $K$-homological aspects but also to the other aspects of spectral triples which are not necessarily homotopy invariant: such as the metric aspect and the link to quantum physics \cite{Chamseddine_Connes_Suijlekom}.}

            
            \item{} In Definition \ref{\pre-spectral triple definition} we have included the assumption \eqref{partials assumption} that $[D^*,a]:\dom(D^*)\to H$ has bounded extension. Since $a:\dom(D^*)\to \dom(D)$, we
                    could have equivalently assumed that $[D,a]:\dom(D)\to H$ has bounded extension. The equivalence of these assumptions is proved in \cite[Lemma 2.1]{Hilsum-2010}.
            {\highlight \item{}\label{commutating_remark} Note the useful identity that if $D = U|D|$ is a polar decomposition, we have $U|D| = |D^*|U$ \cite[Section 7.1]{Schmudgen-2012}, and hence that $Uf(|D|) = f(|D^*|)U$ for all Borel functions $f$.}
        \end{enumerate}
    \end{remark}
    
    {\highlight
    Remark \ref{pre_spectral_triple_remark}.\eqref{commutating_remark} gives us the following useful trick: if $f$ is a Borel function and $D = U|D|$ is a polar decomposition, then:
    \begin{equation}\label{useful_trick_1}
        D^*f(|D^*|) = U^*|D^*|f(|D^*|) = U^*f(|D^*|)|D^*| = f(|D|)U^*|D^*| = f(|D|)D^*.
    \end{equation}
    and similarly:
    \begin{equation}\label{useful_trick_2}
        Df(|D|) = f(|D^*|)D.
    \end{equation}

    Note that we have assumed in Definition \ref{\pre-spectral triple definition} only that $a(1+D^*D)^{-1/2}$ is compact for $a \in \Ac$. The following lemma shows that this assumption is enough to conclude that $a(1+DD^*)^{-1/2}$ is compact.
    The authors wish to extend their gratitude to Jens Kaad for providing us with the following proof:
    \begin{lemma}
        Let $(\Ac,H,D)$ be a pre-spectral triple. Then for all $a \in \Ac$ we have that $a(1+DD^*)^{-1/2}$ is compact.
    \end{lemma}
    \begin{proof}
        We have that $a(1+DD^*)^{-1/2}$ is compact if and only if:
        \begin{equation*}
            a(1+DD^*)^{-1}a^*
        \end{equation*}
        is compact, and indeed it suffices that $a(1+DD^*)^{-1}$ be compact, and so we prove this latter assertion.
    
        Note that we have:
        \begin{equation*}
            \|a(1+D^*D)^{-1}\|_\infty\leq \|a\|_\infty
        \end{equation*}
        and thus for all $x$ in the $C^*$-closure of $\Ac$ we have $x(1+D^*D)^{-1}$ compact. Let $a \in \Ac$ and take $x$ in the $C^*$-closure of $\Ac$.
        Thanks to \eqref{useful_trick_2}, on $\dom(D)$ we have:
        $$D^*D(1+D^*D)^{-1}=D^*(1+DD^*)^{-1}D$$
        Noting that $D\subset D^*,$ we can write
        \begin{align*}
             xD^*D(1+D^*D)^{-1}a(1+DD^*)^{-1} &= xD^*(1+DD^*)^{-1}D^*a(1+DD^*)^{-1}\\
                                              &= xD^*(1+DD^*)^{-1}[D^*,a](1+DD^*)^{-1}\\
                                              &\quad +xD^*(1+DD^*)^{-1}aD^*(1+DD^*)^{-1}.
        \end{align*}
        However, \eqref{useful_trick_1} tells us that $D^*(1+DD^*)^{-1} = (1+D^*D)^{-1/2}D^*(1+DD^*)^{-1/2}$, and therefore it follows that:
        \begin{equation*}
            xD^*(1+DD^*)^{-1}
        \end{equation*}
        is compact.
        
        Therefore:
        \begin{equation*}
            xD^*D(1+D^*D)^{-1}a(1+DD^*)^{-1}
        \end{equation*}
        is compact.
        Thus, 
        \begin{equation*}
            xa(1+DD^*)^{-1} = x(1+D^*D)^{-1}a(1+DD^*)^{-1}+xD^*D(1+D^*D)^{-1}a(1+DD^*)^{-1}
        \end{equation*}
        is compact. However, since $\|xa(1+DD^*)^{-1}\|_\infty \leq \|x\|_\infty\|a\|_\infty$, we can pass to the $C^*$-closure to obtain that:
        \begin{equation*}
            xy(1+DD^*)^{-1}
        \end{equation*}
        is compact for all $x,y$ in the $C^*$-closure of $\Ac$.
        Since every element of a $C^*$-algebra can be factored into two other elements \cite[Proposition 1.4.5]{Pedersen}, we have that $a(1+DD^*)^{-1}$ is compact
    \end{proof}
    
    }
    
    A fundamental geometric property of a \pre-spectral triple is the summability, or dimension.
    \begin{definition}\label{dimension definition}
        Let $p > 0$. We shall say that a \pre-spectral triple is $p$-dimensional if for all $a \in \Ac$ we have:
        \begin{align*}
            a(1+|D|^2)^{-p/2},\; \partial(a)(1+|D|^2)^{-p/2} \in \Lc_{1,\infty},\\
            a(1+|D^*|^2)^{-p/2},\; \partial(a)(1+|D^*|^2)^{-p/2} \in \Lc_{1,\infty}.
        \end{align*}
    \end{definition}
    
    \begin{remark}
        There are some possible variations in the definition of dimension, for example we could have instead required that $a(1+|D|^2)^{-1/2} \in \Lc_{p,\infty}$, and similarly
        with $\partial(a)$ and $D^*$. We have selected Definition \ref{dimension definition} so that the \doubling construction introduced in the next section will yield
        a $p$-dimensional spectral triple (in the sense of \cite{SZ-asterisque}) automatically from a $p$-dimensional \pre-spectral triple.
    \end{remark}

    It is possible to convert a \pre-spectral triple into a genuine spectral triple by a procedure which we call ``\doublinge". The \double of a \pre-spectral triple
    preserves many of the properties of the original \pre-spectral triple, and is defined as follows:
    \begin{definition}\label{clone_def}
        Let $q$ denote the rank $1$ projection:
        \begin{equation*}
            q = \frac{1}{2}\begin{pmatrix} 1 & 1\\ 1 & 1\end{pmatrix}.
        \end{equation*}
    
        The \double of a \pre-spectral triple $(\Ac,H,D)$ is the triple $(\Ac \otimes q,H\otimes \Cplx^2,D_2)$, where
        $D_2$ is the operator:
        \begin{equation*}
            D_2 = \begin{pmatrix} 0 & D^* \\ D & 0\end{pmatrix}.
        \end{equation*}
        with domain $\dom(D)\oplus \dom(D^*)$.
        
        If $(\Ac,H,D)$ has a grading $\gamma$, then $H\otimes \Cplx^2$ is equipped with the grading $\gamma\otimes 1$.
    \end{definition}
    Since $q$ is a projection it is immediate that $\Ac\otimes q$ is an algebra.
    
    The purpose of the \doubling construction is to produce a genuine spectral triple from a \pre-spectral triple. The following proposition verifies this, and shows that the elementary properties of a \pre-spectral triple are reflected in its \doublee.
    \begin{proposition}\label{double is spectral triple}
        Let $(\Ac,H,D)$ be a \pre-spectral triple. Then,
        \begin{enumerate}[{\rm (i)}]
            \item{}\label{D_2 is self-adjoint} $D_2$ is self-adjoint.
            \item{}\label{double respects partials} For all $a \in \Ac$, we have that $a\otimes q$ maps $\dom(D_2)$ into $\dom(D_2)$, and the operator $[D_2,a\otimes q]:\dom(D_2)\to H\otimes\Cplx^2$ has bounded extension equal to $\partial(a)\otimes q$.
            \item{}\label{double respects compact resolvents} The \double $(\Ac\otimes q,H\otimes \Cplx^2,D_2)$ is a spectral triple.
            \item{}\label{double respects dimension} If $(\Ac,H,D)$ is $p$-dimensional then $(\Ac\otimes q,H\otimes\Cplx^2,D_2)$ is $p$-dimensional.
            \item{}\label{double respects gradings} If $(\Ac,H,D)$ is even with grading $\gamma$, then $(\Ac\otimes q,H\otimes \Cplx^2,D_2)$ is even with grading $\gamma\otimes 1$.
        \end{enumerate}
    \end{proposition}
    \begin{proof}
        To prove \eqref{D_2 is self-adjoint}, consider $D_2^*$:
        \begin{equation*}
            D_2^{*} = \begin{pmatrix} 0 & D^*\\ D^{**} & 0 \end{pmatrix}.
        \end{equation*}
        Since $D$ is symmetric and closed, we have that $D^{**} = \overline{D}$ (\cite[Theorem VIII.1.b]{Reed-Simon-I-1980}), and since $D$ is closed, $D^{**} = D$, and so $D_2$ is 
        indeed self-adjoint.
    
        Now we prove \eqref{double respects partials}. Since $a:\dom(D^*)\to\dom(D)$, it is clear that $a\otimes q$ maps $\dom(D_2) = \dom(D)\oplus\dom(D^*)$ to $\dom(D_2)$. On $\dom(D)\oplus \dom(D^*)$ we have:
        \begin{equation*}
            [D_2,a\otimes q] = \frac{1}{2}\begin{pmatrix} D^*a-aD & D^*a-aD^*\\ Da-aD & Da-aD^* \end{pmatrix}.
        \end{equation*}
        Since $D$ is symmetric, for $\xi \in \dom(D)$ we have that $D\xi = D^*\xi$. Since $a$ takes $\dom(D^*)$ to $\dom(D)$ (and so also $\dom(D)$ to $\dom(D)$), the diagonal terms $D^*a-aD$
        and $Da-aD^*$ coincide with $[D,a]$ and $[D^*,a]$ on $\dom(D)$ and $\dom(D^*)$ respectively. Hence on $\dom(D)\oplus\dom(D^*)$ we have:
        \begin{equation*}
            [D_2,a\otimes q] = \frac{1}{2}\begin{pmatrix} [D,a] & [D^*,a]\\ [D,a] & [D^*,a]\end{pmatrix}.
        \end{equation*}
        By assumption, the commutator $[D^*,a]:\dom(D^*)\to H$ has bounded extension $\partial(a)$. Since $a:\dom(D)\to \dom(D)$, the operator $[D,a]:\dom(D)\to H$ is a restriction
        to $\dom(D)$ of a bounded operator $\partial(a)$. So indeed $[D_2,a\otimes q] = \partial(a)\otimes q$ on $\dom(D_2) = \dom(D)\oplus \dom(D^*)$, thus proving \eqref{double respects partials}.
        
        Now we prove \eqref{double respects compact resolvents}. We have already proved in parts \eqref{D_2 is self-adjoint} and \eqref{double respects partials} that $D_2$ is self-adjoint, $a\otimes q:\dom(D_2)\to\dom(D_2)$ and that $[D_2,a\otimes q]$ has bounded extension. What remains is
        to show that $(a\otimes q)(1+D_2^2)^{-1/2}$ is compact. By definition we have:
        \begin{equation*}
            (a\otimes q)(1+D_2^2)^{-1/2} = \frac{1}{2}\begin{pmatrix} a(1+D^*D)^{-1/2} & a(1+DD^*)^{-1/2}\\ a(1+D^*D)^{-1/2} & a(1+DD^*)^{-1/2} \end{pmatrix}
        \end{equation*}
        so indeed $a(1+D_2)^{-1/2}$ is compact if and only if each entry is compact, and thus \eqref{double respects compact resolvents} is proved. Also, \eqref{double respects dimension} follows
        immediately from the preceding display.
        
        Finally, to prove \eqref{double respects gradings}, it is trivial that for $a \in \Ac$, $a\otimes q$ commutes with $\gamma\otimes 1$, and that if $\gamma$ anticommutes with $D$, then it also
        anticommutes with $D^*$ since $\gamma^* = \gamma$. It is then easily verified that $\gamma\otimes 1$ anticommutes with $D_2$.
    \end{proof}
    
    The process of \doubling converts a \pre-spectral triple into a spectral triple (with self-adjoint $D$). This procedure
    is something like a ``completion". It is worthwhile to note that if we apply the \doubling procedure to a \pre-spectral triple
    $(\Ac,H,D)$ where $D$ is self-adjoint (that is, a spectral triple), then we get nothing new, in the sense that the resulting spectral
    triple is unitarily equivalent to $(\Ac\oplus 0,H\oplus H,-D\oplus D)$.
    
    To see this, consider $D_2$ when $D = D^*$,
    \begin{equation*}
        D_2 = \begin{pmatrix} 0 & D\\ D & 0 \end{pmatrix}
    \end{equation*}
    {\highlight
    and the unitary map $W:H\oplus H\to H\oplus H$ given by $W = \frac{1}{\sqrt{2}}\begin{pmatrix} 1 & 1 \\ -1 & 1 \end{pmatrix}$. Then,
    for $\xi \in H$,
    \begin{equation*}
        WD_2\begin{pmatrix} \xi \\ 0\end{pmatrix} = W\begin{pmatrix}0 \\ D\xi\end{pmatrix} = \frac{1}{\sqrt{2}}\begin{pmatrix} D\xi \\ D\xi\end{pmatrix} = \begin{pmatrix}D & 0 \\ 0 & -D\end{pmatrix}W\begin{pmatrix} \xi \\ 0 \end{pmatrix}
    \end{equation*}
    and
    \begin{equation*}
        (a\otimes q)W = \frac{1}{\sqrt{2}}\begin{pmatrix} a & 0 \\ -a & 0 \end{pmatrix} = W\begin{pmatrix}a & 0 \\ 0 & 0\end{pmatrix} = W(a\oplus 0).
    \end{equation*}
    So the map $W$ effects a unitary equivalence between the spectral triples
    $$(\Ac\otimes q,H\otimes \Cplx^2,D_2)$$
    and 
    $$(\Ac\oplus 0 ,H\otimes \Cplx^2,D\oplus -D).$$ }
    In this latter triple, since $\Ac$ acts trivially on the second component, we have essentially the same thing as $(\Ac,H,D)$. In particular, these two triples will define the same element in $K$-homology.
    It is in this sense that the \doubling procedure is ``idempotent".
        
\subsection{Smoothness of \pre-spectral triples}\label{properties subsection}
    Like spectral triples, \pre-spectral triples can be further described by their smoothness, or regularity.
    
    First, if $E$ is a closed operator, we define:
    \begin{equation*}
        \dom_{\infty}(E) = \bigcap_{n > 1} \dom(E^n).
    \end{equation*}
    We note that if $E$ is self-adjoint, then $\dom_{\infty}(E)$ is dense, as it contains at least $\dom(e^{|E|^2})$, which is dense (being the domain of a self-adjoint operator).
    
    In close parallel to the theory for spectral triples, smoothness is defined in terms of domains of certain commutators (c.f. \cite[Section 10.3]{GVF}, \cite[Section 2.3]{CGRS2}, \cite[Definition 2.2(iv)]{CGPRS}).
    \begin{definition}\label{def_R^k}
        Let $E$ be a self-adjoint operator on a Hilbert space $H$. If $T$ is a bounded operator on $H$ which maps $\dom_\infty(E)$ 
        to $\dom_\infty(E)$, then we define:
        \begin{equation*}
            R_E(T) = [E^2,T](1+E^2)^{-1/2}:\dom_\infty(E)\to H.
        \end{equation*}
        We define $\dom(R_E)$ to be the set of bounded operators $T$ on $H$ which map $\dom_\infty(E)$ to $\dom_\infty(E)$ and such that $R_E(T)$ has bounded extension, and we use the same symbol $R_E(T)$ for the bounded extension.
        
        Taking $R_E^1 = R_E$, we then define by induction:
        \begin{equation*}
            \dom(R_E^{k+1}) := \{T \in \dom(R_E^k)\;:\;R_E(T) \in \dom(R_E^k)\},\quad k\geq 1.
        \end{equation*}
        For $T \in \dom(R_E^k)$, $R_E^k(T)$ denotes the bounded extension of $(R_E\circ\cdots R_E)(T)$, where $R_E$ is composed $k$ times.
        
        Finally we set:
        \begin{equation*}
            \dom_\infty(R_E) := \bigcap_{k\geq 1} \dom(R_E^k).
        \end{equation*}
    \end{definition}
    
    We now can define the notion of smoothness for a \pre-spectral triple. The following definition is chosen so that the \double of a smooth \pre-spectral triple is smooth.
    \begin{definition}\label{smoothness definition}
        A \pre-spectral triple $(\Ac,H,D)$ is called smooth if for all $a \in \Ac$ and $n\geq 0$ we have that $a$ and $\partial(a)$ map $\dom((D^*)^n)$ into $\dom(D^n)$, and furthermore:
        \begin{equation*}
            a,\,\partial(a) \in \dom_\infty(R_{|D|})\cap \dom_\infty(R_{|D^*|}),
        \end{equation*}
    \end{definition}
    Of course, if $(\Ac,H,D)$ is a spectral triple (with $D = D^*$) then Definition \ref{smoothness definition} recovers the usual definition of smoothness (as in, e.g., \cite[Definition 2.27]{SZ-asterisque} and \cite[Appendix B]{Connes-Moscovici}).

    By assumption, if $(\Ac,H,D)$ is smooth, and $a \in \Ac$, we have automatically that $R_{|D^*|}(a)$ maps $\dom_\infty(|D^*|)$
    to $\dom_\infty(|D^*|)$. However in general more can be said, since we know the extra property that $a:\dom((D^*)^n)\to\dom(D^n)$. As
    the following lemma shows, $R_{|D^*|}(a)$ maps $\dom_\infty(|D^*|)$ not just to $\dom_\infty(|D^*|)$, but to the smaller subspace $\dom_\infty(D)$:
    \begin{lemma}\label{mixed smoothness conditions}
        Let $(\Ac,H,D)$ be a smooth \pre-spectral triple. Then, for all $x \in \Ac\cup\partial(\Ac)$, and for all $k\geq 0$,
        $R^k_{|D^*|}(x)$ maps $\dom_{\infty}(|D^*|)$ to $\dom_\infty(D)$, and similarly $R^k_{|D|}(x)$ maps $\dom_{\infty}(|D|)$ to $\dom_\infty(D)$.
    \end{lemma}
    \begin{proof}   
        Take $\xi \in \dom_\infty(|D^*|)$. We have:
        \begin{equation*}
            R^k_{|D^*|}(x)\xi = \sum_{j=0}^k (-1)^j\binom{k}{j}|D^*|^{2k-2j}x|D^*|^{2j}(1+|D^*|^2)^{-k/2}\xi.
        \end{equation*}
        We will show that each summand is in $\dom_{\infty}(D)$. For each $j$, we have that:
        \begin{equation*}
            |D^*|^{2j}(1+|D^*|^2)^{-k/2}\xi \in \dom_{\infty}(|D^*|).
        \end{equation*}
        Then, since $(\Ac,H,D)$ is smooth, $x$ maps $\dom_{\infty}(|D^*|)$ to $\dom_{\infty}(D)$. Now, since $|D^*|^{2k-2j} = (DD^*)^{k-j}$, $|D^*|^{2k-2j}$ acts as
        $D^{2k-2j}$ on $\dom_{\infty}(D)$. Also,
        \begin{equation*}
            |D^*|^{2k-2j}x|D^*|^{2j}(1+|D^*|^2)^{-k/2}\xi \in \dom_{\infty}(D).
        \end{equation*}
        Thus $R^k_{|D^*|}(x)\xi \in \dom_{\infty}(D)$, and this proves the first claim.
        The argument for $R^{k}_{|D|}(x)$ is similar, but instead uses the property that $x$ maps $\dom_{\infty}(|D|)$ into $\dom_{\infty}(D)$.
    \end{proof}

    As expected, if $(\Ac,H,D)$ is a smooth \pre-spectral triple then its \double is smooth:
    \begin{theorem}\label{double of smooth is smooth}
        Let $(\Ac,H,D)$ be a smooth \pre-spectral triple. Then:
        \begin{enumerate}[{\rm (i)}]
            \item{}\label{double respects lambdas} For all $x \in \Ac\cup \partial(\Ac)$ we have that $x\otimes q \in \dom_\infty(R_{D_2})$, and for all $k\geq 1$:
                    \begin{equation}\label{formula for double R^k}
                        R^k_{D_2}(x\otimes q) = \frac{1}{2}\begin{pmatrix} R^k_{|D|}(x) & R^k_{|D^*|}(x) \\ R^k_{|D|}(x) & R^k_{|D^*|}(x) \end{pmatrix}.
                    \end{equation}
            \item{}\label{double respects smoothness} The \double $(\Ac\otimes q,H\otimes\Cplx^2,D_2)$ is a smooth spectral triple.
            \item{}\label{double respects 1.2.1} Let $\mathcal{E}$ be an ideal of compact operators (such as the ideal $\Lc_{p,\infty}$), let $x \in \Ac\cup\partial(\Ac)$ and let $k\geq 1$. Then 
                    $R^k_{D_2}(x\otimes q)(1+D_2)^{-1/2} \in \mathcal{E}$
                    if and only if both $R_{|D|}^k(x)(1+|D|^2)^{-1/2}$ and $R^k_{|D^*|}(x)(1+|D^*|^2)^{-1/2}$ are in $\mathcal{E}$.
        \end{enumerate}
    \end{theorem}
    \begin{proof}
        First we show \eqref{double respects lambdas}. Let $x \in \Ac\cup \partial(\Ac)$. We have that:
        \begin{equation*}
            D_2^2 = \begin{pmatrix} D^*D & 0 \\ 0 & DD^* \end{pmatrix}
        \end{equation*}
        and
        \begin{equation*}
            \dom(D_2^2) = \dom(D^*D)\oplus \dom(DD^*).
        \end{equation*}
        Therefore:
        \begin{equation*}
            \dom_{\infty}(D_2) = \dom_{\infty}(|D|)\oplus \dom_{\infty}(|D^*|).
        \end{equation*}
        By the definition of smoothness (Definition \ref{smoothness definition}), $x$ maps $\dom_{\infty}(|D^*|)$ and $\dom_{\infty}(|D|)$ into $\dom_{\infty}(D)$, and hence $x\otimes q$ maps $\dom_{\infty}(D_2)$ into
        \begin{equation*}
            \dom_{\infty}(D)\oplus \dom_{\infty}(D) \subseteq \dom_{\infty}(D_2).
        \end{equation*}
        
        We now proceed to showing that $x\otimes q \in \dom(R^k_{D_2})$ and that \eqref{formula for double R^k} holds by induction on $k$, with the base case $k=0$ being
        immediate. Now supposing \eqref{formula for double R^k} is true for $k\geq 0$ we prove it for $k+1$. We begin by showing that $R^k_{D_2}(x\otimes q)\in \dom(R_{D_2})$.
        
        First, by assumption $x\in \dom_{\infty}(R_{|D|})\cap \dom_{\infty}(R_{|D^*|})$, so by definition:
        \begin{equation*}
            R^k_{|D|}(x):\dom_{\infty}(|D|)\to \dom_{\infty}(|D|),\quad R^k_{|D^*|}(x):\dom_{\infty}(|D^*|)\to \dom_{\infty}(|D^*|)
        \end{equation*}
        and thus by our assumption that \eqref{formula for double R^k} holds and since $${\dom_{\infty}(D_2)=\dom_{\infty}(|D|)\oplus \dom_{\infty}(|D^*|)},$$ we have that
        \begin{equation*}
            R^k_{D_2}(x):\dom_{\infty}(D_2)\to \dom_{\infty}(D_2).
        \end{equation*}
        Now we show that $R_{D_2}(R^k_{D_2}(x\otimes q))$ has bounded extension. On $\dom_{\infty}(D_2)$, by definition we have:
        \begin{align*}
            &R_{D_2}(R^k_{D_2}(x\otimes q)) = [D_2^2,R^k_{D_2}(x\otimes q)](1+D_2^2)^{-1/2}\\
                                           &= \frac{1}{2}\begin{pmatrix} 
                                                                D^*DR^k_{|D|}(x) - R^k_{|D|}(x)D^*D &  D^*DR^k_{|D^*|}(x) - R^k_{|D^*|}(x)DD^* \\
                                                                DD^*R^k_{|D|}(x) - R^k_{|D|}(x)D^*D &  DD^*R^k_{|D^*|}(x) - R^k_{|D^*|}(x)DD^*
                                                         \end{pmatrix}\\
                                           &\quad \quad \cdot (1+D_2^2)^{-1/2}\\
                                           &= \frac{1}{2}\begin{pmatrix}
                                                                R^{k+1}_{|D|}(x) & \mathbf{A}_1\\
                                                                \mathbf{A}_2 & R^{k+1}_{|D^*|}(x)
                                                         \end{pmatrix}
        \end{align*}
        where 
        \begin{align*}
            \mathbf{A}_1 &= (D^*DR^k_{|D^*|}(x)-R^k_{|D^*|}(x)DD^*)(1+|D^*|^2)^{-1/2},\text{ and }\\
            \mathbf{A}_2 &= (DD^*R^k_{|D|}(x)-R^k_{|D|}(x)D^*D)(1+|D|^2)^{-1/2}.
        \end{align*}
        By Lemma \ref{mixed smoothness conditions} we have that $R^k_{|D^*|}(x)$ maps $\dom_{\infty}(|D^*|)$ into $\dom_{\infty}(D)$. Using the fact that $D^*D$ and $DD^*$ coincide on $\dom_\infty(D)$ we have the following for $\xi \in \dom_{\infty}(|D^*|)$:
        \begin{equation*}
            D^*DR^k_{|D^*|}(x)\xi = D^2R^k_{|D^*|}(x)\xi = DD^*R^k_{|D^*|}(x)\xi,
        \end{equation*}
        and similarly for $\xi \in \dom_{\infty}(|D|)$:
        \begin{equation*}
            DD^*R^k_{|D|}(x)\xi = D^2R^k_{|D|}(x)\xi = D^*DR^k_{|D|}(x)\xi.
        \end{equation*}
        Hence, on $\dom_\infty(D_2)=\dom_{\infty}(|D|)\oplus \dom_{\infty}(|D^*|)$, we have:
        \begin{equation*}
            R^{k+1}_{D_2}(x\otimes q) = \frac{1}{2}\begin{pmatrix} R^{k+1}_{|D|}(x) & R^{k+1}_{|D^*|}(x) \\ R^{k+1}_{|D|}(x) & R^{k+1}_{|D^*|}(x) \end{pmatrix}.
        \end{equation*}
        By assumption, each entry has bounded extension, and so finally $R^{k}_{D_2}(x\otimes q)\in \dom(R_{D_2})$, thus $x\otimes q \in \dom(R^{k+1}_{D_2})$ and \eqref{formula for double R^k}
        holds for $k+1$. This proves the inductive step and so \eqref{double respects lambdas} is proved.
        
        It follows from \eqref{double respects lambdas} that if $x \in \Ac\cup \partial(\Ac)$ and $(\Ac,H,D)$ is smooth, then $x\otimes q \in \dom_{\infty}(R_{D_2})$, and so by definition $(\Ac\otimes q,H\otimes\Cplx^2,D_2)$ is smooth, thus
        proving \eqref{double respects smoothness}.
        
        Finally, to prove \eqref{double respects 1.2.1}, we simply use \eqref{double respects lambdas} to get:
        \begin{equation*}
            R^k_{D_2}(x\otimes q)(1+D_2^2)^{-1/2} = \frac{1}{2}\begin{pmatrix} R_{|D|}^k(x)(1+|D|^2)^{-1/2} & R^k_{|D^*|}(x)(1+|D^*|^2)^{-1/2} \\ R^k_{|D|}(x)(1+|D|^2)^{-1/2} & R^k_{|D^*|}(x)(1+|D^*|^2)^{-1/2} \end{pmatrix}
        \end{equation*}
        and \eqref{double respects 1.2.1} immediately follows.
    \end{proof}
    
    In the setting of non-unital spectral triples, it is generally not sufficient to treat smoothness and summability separately,
    as generally we require not only that $x(1+D^2)^{-1/2}$ is compact but for all $k$ that $R^{k}_D(x)(1+D^2)^{-1/2}$ be compact. In the non-unital self-adjoint setting
    assumptions of this nature are common: for example the notion of smooth summability in \cite[Definition 2.19]{CGRS2} and the additional conditions required in \cite[Hypothesis 1.2.1]{SZ-asterisque}.
    
    In the setting of \pre-spectral triples, we have found that the following definition is appropriate:
    \begin{definition}\label{def_smoothly_p_dim}
        A smooth $p$-dimensional \pre-spectral triple is called {\bf smoothly $p$-dimensional} if for all $k\geq 0$ and $x \in \Ac\cup\partial(\Ac)$, we have:
        \begin{equation*}
            R_{|D^*|}^k(x)(1+|D^*|^2)^{-p/2} \in \Lc_{1,\infty}
        \end{equation*}
        and
        \begin{equation*}
            R_{|D|}^k(x)(1+|D|^2)^{-p/2} \in \Lc_{1,\infty}.
        \end{equation*}
    \end{definition}
    
    An alternative definition of being smoothly $p$-dimensional would be to assert that $R^k_{|D^*|}(x)(1+|D^*|^2)^{-1/2}$ is in $\Lc_{p,\infty}$ (and similarly with $D$ in place if $D^*$).
    The following lemma shows how our chosen definition implies this alternative definition.
    \begin{lemma}\label{strong implies weak dimension}
        If $(\Ac,H,D)$ is a smoothly $p$-dimensional \pre-spectral triple with $p \geq 1$ then for all $k\geq 0$ we have:
        \begin{equation*}
            R_{|D^*|}^k(x)(1+|D^*|^2)^{-1/2} \in \Lc_{p,\infty}
        \end{equation*}
        and
        \begin{equation*}
            R_{|D|}^k(x)(1+|D|^2)^{-1/2} \in \Lc_{p,\infty}.
        \end{equation*}
    \end{lemma}
    \begin{proof}
        Since $p\geq 1$, we may apply the Araki-Lieb-Thirring inequality \eqref{ALT inequality},
        \begin{equation*}
            |R_{|D^*|}^k(x)(1+|D^*|^2)^{-1/2}|^p \prec\prec_{\log} |R^{k}_{|D^*|}(x)|^p(1+|D^*|^2)^{-p/2}.
        \end{equation*}
        By the definition of being smoothly $p$-dimensional, we have that $|R^{k}_{|D^*|}(x)|^p(1+D^*D)^{-p/2} \in \Lc_{1,\infty}$, and hence by \eqref{monotonicity}, it follows that $R_{|D^*|}^k(x)(1+|D^*|^2)^{-1/2} \in \Lc_{p,\infty}$.
        
        We can give an identical argument with $|D|$ in place of $|D^*|$ to arrive at the second assertion.
    \end{proof}

\section{Sufficient conditions on a spectral triple to apply the Character Formula}\label{sufficient conditions section}
    Having described how a \pre-spectral triple can be converted into a true spectral triple by the \doubling procedure, our main goal is to state the Character Formula
    for the \doublee. To actually state the Character Formula for a (genuine) spectral triple $(\Bc,K,T)$ requires certain assumptions on $(\Bc,K,T)$.
    
    We follow the prescription given in \cite{SZ-asterisque}. For a spectral triple $(\Bc,K,T)$, $\delta(x)$ denotes the bounded extension of the commutator $[|T|,x]$ when it exists.    
    The following is identical to \cite[Hypothesis 1.2.1]{SZ-asterisque}.
    \begin{hypothesis}\label{main assumption}
        The spectral triple $(\Bc,K,T)$ satisfies the following assumptions:
        \begin{enumerate}[{\rm (i)}]
            \item{}\label{ass0} $(\Bc,K,T)$ is smooth.
            \item{}\label{ass1} There is a positive integer $p\geq 1$ such that $(\Bc,K,T)$ is $p$-dimensional. That is, for $x \in \Bc\cup\partial(\Bc)$ we have $x(T+i)^{-p} \in \Lc_{1,\infty}$.
            \item{}\label{ass2} For all $x \in \Bc\cup\partial(\Bc)$ and all $k\geq 0$, we have the following asymptotic:
            \begin{equation*}
                \|\delta^k(x)(T+i\lambda)^{-1}\|_{\Lc_1} = O(\lambda^{-1})\,\quad \lambda\to\infty.
            \end{equation*} 
        \end{enumerate}
    \end{hypothesis}
    The emphasis in \cite{SZ-asterisque} was to work with minimal assumptions. If one is willing to make more stringent assumptions on $(\Bc,K,T)$, then it is possible
    to provide more easily verified sufficient conditions for Hypothesis \ref{main assumption} to hold. This is what we provide in the following two results.
    
    Denote $U := \chi_{[0,\infty)}(T) - \chi_{(-\infty,0)}(T)$ so that $U$ is unitary and $T = U|T|$, then define $T_0 = U(1+T^2)^{1/2}$. It is proved in \cite[Proposition 3.1.4]{SZ-asterisque}
    that $(\Bc,K,T_0)$ is a spectral triple satisfying Hypothesis \ref{main assumption} if and only if $(\Bc,K,T)$ is. Let $\partial_0(x)$ and $\delta_0(x)$ denote the bounded extension
    of the commutators $[T_0,x]$ and $[|T_0|,x]$ respectively, when they exist. 
    

    Smoothness of a spectral triple may be equivalently defined in terms of $\delta_0$ or $R_T$. This equivalence first appeared in \cite[Appendix B]{Connes-Moscovici}. The following lemma shows
    a related result, although a technicality resulting from the lack of a Banach norm on $\Lc_{1,\infty}$ and a corresponding Bochner integration theory case makes the proof in the $p=1$ case quite technical,
    and so we will defer the proof to Appendix \ref{lambda to delta appendix}.
    \begin{lemma}\label{lambda to delta}
        Let $(\Bc,K,T)$ be a smooth spectral triple, and let $p \geq 1$, and let $x \in \Bc\cup \partial_0(\Bc)$. Then:
        \begin{equation}\label{lambda condition}
            R_{T}^k(x)(1+T^2)^{-1/2} \in \Lc_{p,\infty}
        \end{equation}
        for all $k\geq 0$ if and only if
        \begin{equation}\label{delta condition}
            \delta_0^k(x)(1+T^2)^{-1/2} \in \Lc_{p,\infty}
        \end{equation}
        for all $k\geq 0$.        
    \end{lemma}    
    
    Now we state sufficient conditions on a spectral triple for Hypothesis \ref{main assumption} to hold. The crucial algebraic assumption on $\Bc$ is that it satisfies
    the ``factorisation" property that $\Bc = \Bc\cdot\Bc$. This is easily verified in the geometric examples we consider in Section \ref{model examples section}, but was avoided in \cite{SZ-asterisque} in favour
    of greater generality. The application of this property in noncommutative geometry can also be seen in \cite{CGPRS}.
    \begin{theorem}\label{sufficient conditions for 1.2.1}
        Let $(\Bc,K,T)$ be a smooth spectral triple such that:
        \begin{enumerate}[{\rm (i)}]
            \item{} There exists $p \in \Ntrl$ such that for all $k \geq 0$ and $x \in \Bc\cup \partial_0(\Bc),$
                    \begin{equation*}
                        \delta_0^k(x)(1+T^2)^{-1/2} \in\mathcal{L}_{p,\infty}
                    \end{equation*}
            \item{} Every $a\in\Bc$ can be written as $a=a_1a_2$ with $a_1,a_2\in\mathcal{B},$
        \end{enumerate}
        Then, for all $k\geq 0$ and $x \in \Bc\cup \partial_0(\Bc)$ we have:
        \begin{equation*}
            \delta^k_{0}(x)(1+T^2)^{-p/2} \in \Lc_{1,\infty}.
        \end{equation*}
        Moreover, $(\mathcal{B},K,T_0)$ satisfies Hypothesis \ref{main assumption}.
    \end{theorem}
    \begin{proof}
        We will prove by induction that for all $j \in \mathbb{N}$ we have
        $$\delta_0^k(a)|T_0|^{-j}\in\mathcal{L}_{\frac{p}{j},\infty}$$
        and
        $$\delta_0^k(\partial_0(a))|T_0|^{-j}\in\mathcal{L}_{\frac{p}{j},\infty}.$$
        For $j=1,$ this is our assumption.

        Next, suppose the claim is true for $j\geq 1$ and let us prove it for $j+1.$ By the factorisation property we can choose $a_1,a_2\in \Ac$ so that $a = a_1a_2$. Using the Leibniz rule, 
        \begin{align*}
        \delta_0^k(a) &= \delta_0^k(a_1a_2)\\
                      &= \sum_{\substack{l_1,l_2\geq 0\\ l_1+l_2=k}}\binom{k}{l_1}\delta_0^{l_1}(a_1)\delta_0^{l_2}(a_2).
        \end{align*}
        Hence,
        $$\delta_0^k(a)|T_0|^{-j-1}=\sum_{\substack{l_1,l_2\geq 0\\ l_1+l_2=k}}\binom{k}{l_1}\delta_0^{l_1}(a_1)\delta_0^{l_2}(a_2)|T_0|^{-j-1}.$$
        Now using the identity $[|T_0|^{-1},\delta_0^{l_2}(x)]= -|T_0|^{-1}\delta_0^{l_2+1}(x)|T_0|^{-1}$,
        \begin{align*}
            \delta_0^{l_2}(a_2)|T_0|^{-j-1} &= |T_0|^{-1}\delta_0^{l_2}(a_2)|T_0|^{-j}-[|T_0|^{-1},\delta_0^{l_2}(a_2)]|T_0|^{-j}\\
                                            &= |T_0|^{-1}\delta_0^{l_2}(a_2)|T_0|^{-j}+|T_0|^{-1}\delta_0^{l_2+1}(a_2)|T_0|^{-j-1}.
        \end{align*}
        Therefore,
        \begin{align*}
            \delta_0^k(a)|T_0|^{-j-1} &= \sum_{\substack{l_1,l_2\geq 0\\ l_1+l_2=k}}\binom{k}{l_1}\delta_0^{l_1}(a_1)|T_0|^{-1}\delta_0^{l_2}(a_2)|T_0|^{-j}\\
                                      &\quad +\sum_{\substack{l_1,l_2\geq 0\\ l_1+l_2=k}}\binom{k}{l_1}\delta_0^{l_1}(a_1)|T_0|^{-1}\delta_0^{l_2+1}(a_2)|T_0|^{-j-1}.
        \end{align*}
        By the inductive assumption and H\"older's inequality:
        \begin{align*}
                              \delta_0^{l_1}(a_1)|T_0|^{-1}\cdot\delta_0^{l_2}(a_2)|T_0|^{-j} &\in \mathcal{L}_{p,\infty}\cdot\mathcal{L}_{\frac{p}{j},\infty}\subset\mathcal{L}_{\frac{p}{j+1},\infty},\\
            \delta_0^{l_1}(a_1)|T_0|^{-1}\cdot\delta_0^{l_2+1}(a_2)|T_0|^{-j}\cdot |T_0|^{-1} &\in \mathcal{L}_{p,\infty}\cdot\mathcal{L}_{\frac{p}{j},\infty}\cdot\mathcal{L}_{\infty}\subset\mathcal{L}_{\frac{p}{j+1},\infty}.
        \end{align*}
        Hence,
        $$\delta_0^k(a)|T_0|^{-j-1}\in\mathcal{L}_{\frac{p}{j+1},\infty}.$$
        
        By a similar inductive argument, one can prove that
        $$\delta_0^k(\partial_0(a))|T_0|^{-j-1}\in\mathcal{L}_{\frac{p}{j+1},\infty}.$$
        
        Taking $j=p$ yields the first claim. Now we show that Hypothesis \ref{main assumption} is satisfied.

        Taking $k=0$ and $j=p,$ we have:
        $$a|T_0|^{-p}\in\mathcal{L}_{1,\infty},\quad \partial_0(a)|T_0|^{-p}\in\mathcal{L}_{1,\infty}.$$
        Hence,
        $$a(T+i)^{-p}\in\mathcal{L}_{1,\infty},\quad\partial_0(a)(T+i)^{-p}\in\mathcal{L}_{1,\infty}.$$
        So $(\Bc,K,T)$ is $p$-dimensional. This verifies Hypothesis \ref{main assumption}.\eqref{ass1}.

        Now we focus on proving Hypothesis \ref{main assumption}.\eqref{ass2}. By applying the $j=p+1$ case, and using the fact that $\frac{|T_0|}{T_0+i\lambda}$ is bounded, we have:
        \begin{align*}
            \|\delta_0^k(a)(T_0+i\lambda)^{-p-1}\|_{\frac{p}{p+1},\infty} &\leq \|\delta_0^k(a)|T_0|^{-p-1}\|_{\frac{p}{p+1},\infty}\\
                                                                          &= O(1), \quad \lambda \to \infty.
        \end{align*}
        We also have, as $\lambda\to \infty$,
        $$\|\delta_0^k(a)(T_0+i\lambda)^{-p-1}\|_{\infty}\leq\|\delta_0^k(a)\|_{\infty}\|(T_0+i\lambda)^{-1}\|_{\infty}^{p+1}=O(\lambda^{-p-1})$$
        where the final equality follows from the operator inequality $|T_0+i\lambda|^{-1}\leq \lambda^{-1}$.
        
        Using the inequality \eqref{interpolation inequality} it follows that
        \begin{align*}
            \|\delta_0^k(a)(T_0+i\lambda)^{-p-1}\|_1 &\leq c_p\|\delta_0^k(a)(T_0+i\lambda)^{-p-1}\|_{\frac{p}{p+1},\infty}^{\frac{p}{p+1}}\|\delta_0^k(a)(T_0+i\lambda)^{-p-1}\|_{\infty}^{\frac{1}{p+1}}\\
                                                     &= O(1)^{\frac{p}{p+1}}\cdot O(\lambda^{-p-1})^{\frac1{p+1}}\\
                                                     &=O(\lambda^{-1}).
        \end{align*}
        By an identical argument, we also have:
        $$\|\delta_0^k(\partial_0(a))(T_0+i\lambda)^{-p-1}\|_1 = O(\lambda^{-1}).$$
        This verifies Hypothesis \ref{main assumption}.\eqref{ass2} for $(\Bc,K,T_0)$. 
        By \cite[Proposition 3.1.4.(iv)]{SZ-asterisque}, it follows that $(\Bc,K,T)$ also satisfies Hypothesis \ref{main assumption}.
    \end{proof}
    
    The preceding theorem has the following useful corollary in terms of $R_T$:
    \begin{corollary}\label{cor_R_diff_powers_diff_Schatten}
        Let $(\Bc,K,T)$ be a smooth spectral triple such that:
        \begin{enumerate}[{\rm (i)}]
            \item{} There exists $p \in \Ntrl$ such that for all $k \geq 0$ and $x \in \Bc\cup \partial_0(\Bc),$
                    \begin{equation*}
                        R_T^k(x)(1+T^2)^{-1/2} \in\mathcal{L}_{p,\infty}
                    \end{equation*}
            \item{} Every $a\in\Bc$ can be written as $a=a_1a_2$ with $a_1,a_2\in\mathcal{B},$
        \end{enumerate}
        Then, for all $k\geq 0$ and all $x \in \Bc\cup \partial_0(\Bc)$,
        \begin{equation*}
            R^k_T(x)(1+T^2)^{-p/2} \in \Lc_{1,\infty}.
        \end{equation*}
    \end{corollary}
    \begin{proof}
        Due to Lemma \ref{lambda to delta}, we have that for all $k\geq 0$ and $x \in \Bc\cup \partial_0(\Bc)$:
        \begin{equation*}
            \delta_0^k(x)(1+T^2)^{-1/2} \in \Lc_{p,\infty}
        \end{equation*}
        Now directly applying Theorem \ref{sufficient conditions for 1.2.1} yields for all $k\geq 0$ that:
        \begin{equation*}
            \delta_0^k(x)(1+T^2)^{-p/2} \in \Lc_{1,\infty}.
        \end{equation*}
        Recall that $|T_0| = (1+T^2)^{1/2}$. We can now express $R^k_T$ in terms of $\delta_0$ as follows. By the Leibniz rule:
        \begin{align*}
                  R_T(x) &= [|T_0|^2,x]|T_0|^{-1}\\
                         &= [|T_0|,x]+|T_0|[|T_0|,x]|T_0|^{-1}\\
                         &= 2\delta_0(x)+\delta^2(x)|T_0|^{-1}.
        \end{align*}
        Then by the binomial theorem,
        \begin{equation*}
            R_T^k(x) = \sum_{l=0}^k \binom{k}{l} 2^l\delta_0^{2k-l}(x)|T_0|^{-k+l}.
        \end{equation*}
        So multiplying on the right by $|T_0|^{-p}$,
        \begin{equation*}
            R_T^k(x)|T_0|^{-p} = \sum_{l=0}^k \binom{k}{l}2^l\delta_0^{2k-l}(x)|T_0|^{-k+l-p}.
        \end{equation*}
        We have proved that each summand is in $\Lc_{1,\infty}$, and so the result follows.
    \end{proof}
    
    By a short argument using Theorem \ref{double of smooth is smooth}, Lemma \ref{strong implies weak dimension} and Theorem \ref{sufficient conditions for 1.2.1}, we get the following:
    \begin{corollary}\label{sufficient conditions for double to satisfy 1.2.1}
        If $(\Ac,H,D)$ is a smoothly $p$-dimensional \pre-spectral triple such that every $a \in \Ac$ can be factorised as $a = a_1a_2$, where $a_1,a_2 \in \Ac$, then the 
        \double $(\Ac\otimes q,H\otimes \Cplx^2,D_2)$ satisfies Hypothesis \ref{main assumption}.
    \end{corollary}
    \begin{proof}
        First, by Theorem \ref{double of smooth is smooth}, the \double $(\Ac\otimes q,H\otimes \Cplx^2,D_2)$ is a smooth, $p$-dimensional spectral triple. 
        Let $\Bc = \Ac\otimes q$, $K = H\otimes \Cplx^2$ and $T = D_2$. By \cite[Proposition 3.1.4]{SZ-asterisque}, the spectral triple $(\Bc,K,T_0)$ is smooth, $p$-dimensional
        and satisfies Hypothesis \ref{main assumption} if and only if $(\Bc,K,T)$ does.
        
        Using Lemma \ref{strong implies weak dimension}, we have that for all $k\geq 0$ and $x \in \Ac\cup \partial(\Ac)$:
        \begin{align*}
            R^k_{|D^*|}(x)(1+|D^*|^2)^{-1/2} &\in \Lc_{p,\infty},\\
            R^k_{|D|}(x)(1+|D|^2)^{-1/2} &\in \Lc_{p,\infty}.
        \end{align*}
        Now by Theorem \ref{double of smooth is smooth}.\eqref{double respects 1.2.1}, it follows that
        \begin{align*}
            R^k_{D_2}(x\otimes q)(1+D_2^2)^{-1/2} \in \Lc_{p,\infty}.
        \end{align*}
        Thus for all $x \in \Bc\cup \partial(\Bc)$,
        \begin{equation*}
            R^k_{T}(x)(1+T^2)^{-1/2} \in \Lc_{p,\infty}.
        \end{equation*}
        Suppose that $a \in \Bc$. Then,
        \begin{equation*}
            \partial(a)-\partial_0(a) = [U|T|-U(1+T^2)^{1/2},a].
        \end{equation*}
        We have that:
        \begin{equation*}
            U|T|-U(1+T^2)^{1/2} = -\frac{U}{|T|+(1+T^2)^{1/2}}
        \end{equation*}
        is bounded, and $U|T|-U(1+T^2)^{1/2}$ also commutes with $R_T^k$. Hence,
        \begin{equation*}
            R^k_T(\partial(a))-R^k(\partial_0(a)) = [-\frac{U}{(1+T^2)^{1/2}+|T|},R^k_T(a)(1+T^2)^{-1/2}] \in \Lc_{p,\infty}.
        \end{equation*}
        Thus for all $x \in \Bc\cup\partial_0(\Bc)$, 
        \begin{equation*}
            R^k_T(x)(1+T^2)^{-1/2} \in \Lc_{p,\infty}.
        \end{equation*}
        Applying Lemma \ref{lambda to delta}, we therefore have that for all $x \in \Bc\cup \partial_0(\Bc)$, $$\delta_0(x)(1+T^2)^{-1/2} \in \Lc_{p,\infty}.$$
        
        Thus $(\Bc,K,T_0)$ satisfies all of the assumptions of Theorem \ref{sufficient conditions for 1.2.1} and therefore satisfies Hypothesis \ref{main assumption}.
    \end{proof}
    
\section{The Character Theorem for \pre-spectral triples}\label{character formula section}
  
    Let $(\Ac,H,D)$ be a smoothly $p$-dimensional \pre-spectral triple satisfying the conditions of Corollary \ref{sufficient conditions for double to satisfy 1.2.1}. Then by that corollary,
    the \double $(\Ac\otimes q,H\otimes\Cplx^2,D_2)$ is a smooth spectral triple satisfying Hypothesis \ref{main assumption}, and so by direct appeal to \cite{SZ-asterisque} we could
    simply state the Character formula for the \doublee. However it is desirable to express the Character theorem not in terms of the \double but in terms of the original \pre-spectral triple.

    Motivated by our model example (Section \ref{model examples section}), this can be achieved under following assumption:
    \begin{hypothesis}\label{lotoreichik hypothesis}
        $(\Ac,H,D)$ is a smoothly $p$-dimensional \pre-spectral triple, where $p \geq 1$ is an integer.        
        We also assume that $\Ac = \Ac\cdot\Ac$. That is, every $a \in \Ac$ can be written as a product $a = a_1a_2$ for some $a_1,a_2 \in \Ac$.

        If $p > 1$, assume that for all $0 \leq a \in \Ac$,
        \begin{equation*}
            a(1+DD^*)^{-1}a - a(1+D^*D)^{-1}a \in \Lc_{\frac{p}{4},\infty}.
        \end{equation*}
        
        If $p = 1$, we instead assume that:
        \begin{equation*}
            (1+D^*D)^{-1}-(1+DD^*)^{-1} \in \Lc_{\frac{1}{4},\infty}.
        \end{equation*}        
    \end{hypothesis}
    Due to Corollary \ref{sufficient conditions for double to satisfy 1.2.1}, if $(\Ac,H,D)$ satisfies Hypothesis \ref{lotoreichik hypothesis} then its \double is a spectral triple
    satisfying Hypothesis \ref{main assumption}. 
    \begin{remark}
        In the $p > 1$ case we have made the assumption that $a(1+DD^*)^{-1}a-a(1+D^*D)^{-1}a$ is in $\Lc_{p/4,\infty}$ since this can be verified in our model example (Section \ref{model examples section}).
        However for the Character Formula it is possible to weaken this assumption: we could assume that $a(1+DD^*)^{-1}a-a(1+D^*D)^{-1}a \in \Lc_{\frac{p}{2}-\varepsilon,\infty}$ for some $\varepsilon > 0$.
    \end{remark}
    
    With Hypothesis \ref{lotoreichik hypothesis}, we can show the following proposition. We defer the proof to Appendix \ref{operator difference appendix}. 
    The proof is however quite technical, requiring very recent operator inequalities due to Eric Ricard \cite{Ricard-fractional-powers} and operator 
    integration techniques recently developed in \cite{SZ-asterisque}.    
    \begin{theorem}\label{difference is trace class}
        Let $(\Ac,H,D)$ be a \pre-spectral triple satisfying Hypothesis \ref{lotoreichik hypothesis}, then for all $0 \leq a \in \Ac$ we have:
        \begin{equation*}
            (1+D^*D)^{-p/2}a^p-(1+DD^*)^{-p/2}a^p \in \Lc_1.
        \end{equation*} 
    \end{theorem}
        
    The following proposition shows how Theorem \ref{difference is trace class} can be used to express the Character Formula of the \double of a \pre-spectral triple
    in terms of the original \pre-spectral triple. Recall that $D_2 = \begin{pmatrix} 0 & D^*\\ D & 0 \end{pmatrix}$.
    \begin{proposition}\label{character_formula_lhs}
        Let $(\Ac,H,D)$ satisfy Hypothesis \ref{lotoreichik hypothesis} with grading $\gamma$ in the even case.
        Let $a_0\otimes \cdots \otimes a_p \in \Ac^{\otimes (p+1)}$ be such that there exists $0 \leq \phi \in \Ac$ such that $\phi a_0 = a_0$. Let $\varphi$ be a trace on $\Lc_{1,\infty}$. Then:
        \begin{equation*}
            \varphi((\gamma\otimes 1) (a_0\otimes q)\prod_{j=1}^p [D_2,a_j\otimes q](1+D_2^2)^{-p/2}) = \varphi(\gamma a_0\prod_{j=1}^p \partial(a_j)(1+D^*D)^{-p/2}).
        \end{equation*}
        {\highlight In the odd case, the same formula holds with $\gamma = 1$.}
    \end{proposition}
    \begin{proof}
        For all $j = 1,\ldots p$, by Lemma \ref{double is spectral triple}.\eqref{double respects partials} we have that $[D_2,a_j\otimes q] = \partial(a_j)\otimes q$. Thus,
        \begin{align*}
            a_0\otimes q\prod_{k=1}^p [D_2,a_k\otimes q] &= (a_0\otimes q)(\prod_{k=1}^p(\partial(a_j)\otimes q))\\
                                                         &= (a_0\prod_{k=1}^p \partial(a_j))\otimes q^{p+1}.
        \end{align*}
        But $q$ is a projection, so the left hand side in the statement of the lemma simplifies to:
        \begin{equation*}
            \varphi((\gamma a_0\prod_{j=1}^p \partial(a_j))\otimes q(1+D_2^2)^{-p/2})
        \end{equation*}
        and this is:
        \begin{equation*}
             \frac{1}{2}(\varphi(\gamma a_0\prod_{j=1}^p \partial(a_j)(1+D^*D)^{-p/2})+\varphi(\gamma a_0\prod_{j=1}^p \partial(a_j)(1+DD^*)^{-p/2})).
        \end{equation*}
        By assumption there exists $0 \leq \phi \in \Ac$ so that $\phi a_0 = a_0$. Thus we may insert $p$ copies of $\phi$ to get:
        \begin{equation*}
            \frac{1}{2}(\varphi(\gamma \phi^p a_0\prod_{j=1}^p \partial(a_j)(1+D^*D)^{-p/2})+\varphi(\gamma \phi^pa_0\prod_{j=1}^p \partial(a_j)(1+DD^*)^{-p/2}))
        \end{equation*}
        Since $\gamma$ commutes with $\phi^p$, and using the cyclic property of the trace, we have:
        \begin{equation*}
            \frac{1}{2}(\varphi(\gamma a_0\prod_{j=1}^p \partial(a_j)(1+D^*D)^{-p/2}\phi^p)+\varphi(\gamma a_0\prod_{j=1}^p \partial(a_j)(1+DD^*)^{-p/2}\phi^p))
        \end{equation*}
        Now applying Theorem \ref{difference is trace class},
        \begin{equation*}
            (1+D^*D)^{-p/2}\phi^p-(1+DD^*)^{-p/2}\phi^p \in \Lc_1
        \end{equation*}        
        and since $\varphi$ vanishes on $\Lc_1$,
        \begin{equation*}
            \varphi(\gamma a_0\prod_{j=1}^p \partial(a_j)(1+DD^*)^{-p/2}\phi^p) = \varphi(\gamma a_0\prod_{j=1}^p \partial(a_j)(1+D^*D)^{-p/2}\phi^p).
        \end{equation*}
        Removing the factors of $\phi^p$,
        \begin{equation*}
            \varphi(\gamma a_0\prod_{j=1}^p \partial(a_j)(1+DD^*)^{-p/2}) = \varphi(\gamma a_0\prod_{j=1}^p \partial(a_j)(1+D^*D)^{-p/2}).
        \end{equation*}
        Thus,
        \begin{equation*}
            \varphi((\gamma\otimes 1) (a_0\otimes q)\prod_{j=1}^p [D_2,a_j\otimes q](1+D_2^2)^{-p/2}) = \varphi(\gamma a_0\prod_{j=1}^p \partial(a_j)(1+D^*D)^{-p/2}).
        \end{equation*}
    \end{proof}
       
    The Chern character of a spectral triple is defined in terms of a representative of the $K$-homology class defined by $(\Ac\otimes q, H\otimes \Cplx^2,D_2)$. There are several ways
    to select such a representative. Here, we follow the approach of \cite{SZ-asterisque}. Note that for a densely defined closed operator $D$, it is possible
    to define a polar decomposition $D = F|D|$, where the operator $F$ -- a partial isometry mapping $\ker(D)^\perp$ to $\ker(D^*)^{\perp}$ -- is called the phase of $D$ (see \cite[Section 7.1]{Schmudgen-2012}).
    \begin{definition}
        Let $P_{\ker(D)}$ be the projection onto $\ker(D)$, and let $P_{\ker(D^*)}$ be the projection onto $\ker(D^*)$. 
        Let $F$ be the phase of $D$. That is, a partial isometry with $D = F|D|$, initial space $\ker(D)^\perp$ and final space $\ker(D^*)^{\perp}$.
        
        Then, define $\widetilde{F}$ to be the matrix:
        \begin{equation*}
            \widetilde{F} := \begin{pmatrix} 0 & F^* & P_{\ker(D)} & 0\\
                                             F & 0 & 0 & P_{\ker(D^*)}\\
                                             P_{\ker(D)} & 0 & 0 & -F^*\\
                                             0 & P_{\ker(D^*)} & -F & 0
                             \end{pmatrix}.
        \end{equation*}
        We also define a representation $\pi$ of $\Ac$ on the Hilbert space $H\otimes \Cplx^4$ by:
        \begin{equation*}
            \pi(a) = \frac{1}{2}\begin{pmatrix} 
                                    a & a & 0 & 0\\
                                    a & a & 0 & 0\\
                                    0 & 0 & 0 & 0\\
                                    0 & 0 & 0 & 0
                                \end{pmatrix}.
        \end{equation*}
        {\highlight If the spectral triple is even with grading $\gamma$, we also provide a grading for the clone. } Let $\Gamma$ be the grading on $H\otimes \Cplx^4$ given by:
        \begin{equation*}
            \Gamma = \begin{pmatrix}
                         \gamma & 0 & 0 & 0\\
                         0 & \gamma & 0 & 0\\
                         0 & 0 & -\gamma & 0\\
                         0 & 0 & 0 & -\gamma
                     \end{pmatrix}.
        \end{equation*}
    \end{definition}
    
    The definition of $\widetilde{F}$ is chosen so that $\widetilde{F}^2=1$ and we have the following:
    \begin{lemma}
        Let $P_{\ker(D_2)}$ be the projection onto the kernel of the self-adjoint operator $D_2$, and define $F_{D_2} := \chi_{(0,\infty)}(D_2)-\chi_{(-\infty,0)}(D_2)$ We have that:
        \begin{equation*}
            \widetilde{F} = \begin{pmatrix} F_{D_2} & P_{\ker(D_2)}\\ P_{\ker(D_2)} & -F_{D_2} \end{pmatrix}.
        \end{equation*}
        In particular $\widetilde{F}^2=1$.
    \end{lemma}
    \begin{proof}
        The proof relies on the fact that if $F$ is the phase of $D$, then $F^*$ is a phase of $D^*$ (see e.g. \cite[Section 7.1]{Schmudgen-2012}). Then,
        \begin{equation*}
            F_{D_2} = \begin{pmatrix}0 & F^*\\ F & 0\end{pmatrix}.
        \end{equation*}
        Since $|D_2| = |D|\oplus |D^*|$, we also have:
        \begin{equation*}
            P_{\ker(D_2)} = \begin{pmatrix} P_{\ker(D)} & 0 \\ 0 & P_{\ker(D^*)}\end{pmatrix}.
        \end{equation*}
        The verification that $\widetilde{F}^2 = 1$ follows from the fact that $F_{D_2}^2 + P_{\ker(D_2)} = 1$.
    \end{proof}
        
    Adopting the terminology of \cite{SZ-asterisque}, we say that a Hochschild cycle $c \in \Ac^{\otimes (p+1)}$ is local if there exists $0 \leq a \in \Ac$ such that $ac = c$, where $a$
    acts on the first tensor component.
    Then we have, by a direct application of the Character Theorem \cite[Theorem 1.2.4]{SZ-asterisque} and Proposition \ref{character_formula_lhs}, the following version of the Character Theorem for \pre-spectral triples:
    \begin{theorem}\label{character theorem}
        Let $(\Ac,H,D)$ be a \pre-spectral triple {\highlight with grading $\gamma$ in the even case} satisfying Hypothesis \ref{lotoreichik hypothesis}. Let $c \in \Ac^{\otimes (p+1)}$ be a local Hochschild cycle, then for any normalised trace $\varphi$ on $\Lc_{1,\infty}$:
        \begin{equation*}
            \varphi(\Omega(c)(1+D^*D)^{-p/2}) = \Ch(c)
        \end{equation*}
        where,
        \begin{align*}
            \Omega(a_0\otimes\cdots \otimes a_p) &= \gamma a_0\prod_{k=1}^p \partial(a_k),\\
                \Ch(a_0\otimes\cdots\otimes a_p) &= \frac{1}{2}\Tr(\Gamma \widetilde{F}\prod_{k=0}^p [\widetilde{F},\pi(a_k)])
        \end{align*}
	{\highlight In the odd case, we have the same formula but with $\gamma=1$.}
    \end{theorem}
       
    \begin{remark}
        The form of the left hand side in Theorem \ref{character theorem} is exactly as in the usual self-adjoint character theorem, albeit here we have $D^*D$ instead of $D^2$. 
        By Theorem \ref{difference is trace class}, we could just as well have used $DD^*$, since the trace $\varphi$ vanishes on $\Lc_1$. 
    \end{remark}
    
\section{The model example of the Euclidean Dirac operator}\label{model examples section}
    To demonstrate the applicability of our assumptions, we study a ``model example" of a symmetric non-self-adjoint operator. 
    Let $\Omega$ be an arbitrary open subset of $\Rl^d$, with $d > 1$. 
    
    Let $C^\infty_c(\Omega)$ be the algebra of smooth functions with compact support in $\Omega$, we will take $D$ to be the Dirac
    operator on $\Omega$ with Dirichlet boundary conditions (see Subsection \ref{model example definitions}). We then claim that:
    \begin{equation*}
        (C^\infty_c(\Omega),L_2(\Omega,\Cplx^N),D)
    \end{equation*}
    is a smoothly $d$-dimensional \pre-spectral triple satisfying Hypothesis \ref{lotoreichik hypothesis}, where $N = 2^{\lfloor d/2\rfloor}$ and $C^\infty_c(\Omega)$
    acts on $L_2(\Omega,\Cplx^N)$ by pointwise multiplication. 
%
%
\subsection{Function spaces and distributions}\label{sobolev spaces section}    
    The following material concerning Sobolev spaces on domains is well known and may be found in references such as \cite{Adams-Fournier-sobolev-2003, Brezis-fa-2010, Burenkov-sobolev-spaces-1998,shubin}.
    
    Recall that $\Omega$ is an open subset of $\Rl^d$. Let $\alpha \in \Ntrl^d$ be a multi-index, and let $\partial^{\alpha}$ be defined
    on $C^\infty_c(\Omega)$ by:
    \begin{equation*}
        \partial^{\alpha} = \partial_{x_1}^{\alpha_1}\partial_{x_2}^{\alpha_2}\cdots \partial_{x_d}^{\alpha_d}.
    \end{equation*}
    where $\partial_{x_1}$, etc. are the partial derivatives in the coordinate variables.
    The distributional derivative of $u \in L_{1,\loc}(\Omega)$, denoted $\partial^\alpha u$, is defined by:
    \begin{equation*}
        (\partial^{\alpha} u,\phi) = (-1)^{|\alpha|}(u,\partial^\alpha \phi),\quad \text{ for all } \phi \in C^\infty_c(\Omega)
    \end{equation*}
    where $(\eta,\zeta)$ denotes the distributional pairing $\int_{\Rl^d} \eta(t)\zeta(t)\,dt$.
    
    Now fix $N \geq 1$. The space $C^\infty_c(\Omega,\Cplx^N)$ is equipped with a canonical locally convex topology (see e.g. \cite[Section 1.56]{Adams-Fournier-sobolev-2003}), and the space $\Dc'(\Omega,\Cplx^N)$ is defined to be the topological dual of $C^\infty_c(\Omega,\Cplx^N)$, with the weak$^*$-topology.
    
    The Sobolev space $W^{k,p}(\Omega,\Cplx^N)$ is defined to be the space of all $u \in L_p(\Omega,\Cplx^N)$ such that for all multi-indices $\alpha$ with $|\alpha| \leq k$
    the distributional derivative $\partial^{\alpha}u$ is in $L_p(\Omega,\Cplx^N)$, with Sobolev norm defined by:
    \begin{equation*}
         \|u\|_{W^{k,p}(\Omega,\Cplx^N)} := \left(\sum_{|\alpha| \leq k} \|\partial^{\alpha} u\|_{L_p(\Omega,\Cplx^N)}^p\right)^{1/p}.
    \end{equation*}
    
    The space $W^{k,p}_0(\Omega,\Cplx^N)$ is defined to be the closure of $C^\infty_c(\Omega,\Cplx^N)$ in the $W^{k,p}$-norm.
    
    We also consider the local Sobolev space $W^{k,p}_{\loc}(\Omega,\Cplx^N)$, of elements $u \in L_{p,\loc}(\Omega,\Cplx^N)$ such that for all $|\alpha|\leq k$ we have
    $\partial^{\alpha} u \in L_{p,\loc}(\Omega,\Cplx^N)$ (see \cite[Section 2.3]{Burenkov-sobolev-spaces-1998}). 
        
\subsection{Pre-spectral triple for the model example}\label{model example definitions}
    Let $N = 2^{\lfloor d/2\rfloor}$ and let $\{\gamma_1,\ldots,\gamma_d\}$ be $d$-dimensional Euclidean gamma matrices, i.e. a fixed family
    of $N\times N$ complex Hermitian matrices satisfying $\gamma_j\gamma_k + \gamma_k\gamma_j = 2\delta_{j,k}1$, $1\leq j,k\leq d$.
    The distributional Dirac operator $\Dd$ is defined as a linear combination of distributional derivatives:
    \begin{equation*}
        \Dd := \sum_{j=1} -i\gamma_j\otimes \partial_j.
    \end{equation*}
    
    Similarly, the distributional Laplacian $\Delta$ is defined as a sum of distributional derivatives:
    \begin{equation*}
        \Delta  := \sum_{j=1}^d \partial_{x_j}^2 .
    \end{equation*}
    Or equivalently, $1\otimes \Delta = -\Dd^2$.    
%
%
    Our model \pre-spectral triple will be based on the Dirac operator with Dirichlet boundary conditions, defined as:
    \begin{align*}
	\dom(D) := W^{1,2}_0(\Omega,\Cplx^N),\\
	Du = \Dd u,\quad u \in \dom(D).
    \end{align*}
    Our first step is to show that $D$ is closed and symmetric. With this goal in mind, we define the auxiliary operator $D_0$, defined
    as $D_0 = \Dd$ with domain $C^\infty_c(\Omega,\Cplx^N)$ and we prove that $D_0$ is symmetric and closable and that $D$ is the closure of $D_0$.
%
    The following proof rests primarily on 
    Green's identity. Let $F \in W^{1,1}(\Omega,\Cplx^d)$ be supported in a compact subset of $\Omega$. Then,
    \begin{equation}\label{Green div}
        \int_{\Omega} \sum_{j=1}^d (\partial_jF_j)(x)\,dx = 0.
    \end{equation}
    Note that no assumptions on the boundary $\partial\Omega$ are needed for the above identity to hold: since $F$ is smooth and compactly supported this can be proved with an application of Fubini's theorem.

    \begin{lemma}\label{lemma_D_0_symmetric}
        The operator $D_0$ defined as $\Dd$ with domain $C^\infty_c(\Omega,\Cplx^N)$ is symmetric (and hence, closable) and the graph norm $\|\cdot\|_{\Gamma(D_0)}$ of $D_0$ is equivalent to the Sobolev norm $\|\cdot\|_{W^{1,2}(\Omega, \Cplx^N)}.$
    \end{lemma}   
    \begin{proof}
        To prove the first assertion let $f,g\in \dom(D_0)$ be arbitrary. By the definition of $D_0$ we have 
        \begin{align*}
            \langle D_0f,g\rangle_{L_2(\Omega,\Cplx^n)}  &= -i\sum_{j=1}^d \int_\Omega \langle \gamma_j \cdot \partial_j f(\bx), g(\bx)\rangle_{\Cplx^N}d\bx.
        \end{align*}
        For every fixed $j=1,\dots, d$, using the Leibniz rule and the fact that $\gamma_j$ is unitary we obtain
        \begin{align*}
            \langle \gamma_j \cdot \partial_j f(\bx), g(\bx)\rangle_{\Cplx^N} &= \partial_j(\langle f(\bx), \gamma_j\cdot g(\bx)\rangle_{\Cplx^N})-\langle f(\bx), \partial_j(\gamma_j\cdot  g)(\bx)\rangle_{\Cplx^N}\\
                                                                              &= \partial_j(\langle f(\bx), \gamma_j\cdot g(\bx)\rangle_{\Cplx^N})-\langle f(\bx), \gamma_j\cdot (\partial_j g)(\bx)\rangle_{\Cplx^N}
        \end{align*}
        Hence, 
        \begin{align*}
            \langle D_0f,g\rangle_{L_2(\Omega,\Cplx^N)} &= -i\sum_{j=1}^d \Big(\int_\Omega \partial_j(\langle f(\bx), \gamma_j\cdot g(\bx)\rangle_{\Cplx^N})d\bx\\
                                                        &\quad +i\int_\Omega \langle f(\bx), \gamma_j\cdot  (\partial_j g)(\bx)\rangle_{\Cplx^N}d\bx\Big)\\
                                                        &=\int_\Omega \sum_{j=1}^d \partial_j\langle f(\bx), \gamma_j\cdot g(\bx)\rangle_{\Cplx^N})d\bx+\langle f, D_0 g\rangle_{L_2(\Omega, \Cplx^N)}.
        \end{align*}
        Thus, to show that the operator $D_0$ is symmetric, it is sufficient to show that 
        $$\int_\Omega \sum_{j=1}^d \partial_j\langle f(\bx), \gamma_j\cdot g(\bx)\rangle_{\Cplx^N}d\bx=0.$$
        To this end consider the vector field $\bbF$ defined on $\bx \in \Rl^d$ by:
        $\bbF(\bx)=\{\langle f(\bx),\gamma_j\cdot g(\bx)\rangle_{\Cplx^N}\}_{j=1}^d.$
        However, by assumption, the functions $f$ and $g$ are in $C^\infty_c(\Omega,\Cplx^N)$, and so $\bbF \in C^\infty_c(\Omega,\Cplx^d)$. Thus by \eqref{Green div},
        \begin{align*}
            \int_{\Omega} (\nabla\cdot \bbF)(\bx) \,d\bx &= \int_\Omega \sum_{j=1}^d \partial_j\langle f(\bx), \gamma_j\cdot g(\bx)\rangle_{\Cplx^N})d\bx\\
                                                         &=0,
        \end{align*}
        as required. 

        Next, to prove that the graph norm $\|\cdot\|_{\Gamma(D_0)}$ is equivalent to the Sobolev norm ${\|\cdot\|_{W^{1,2}(\Omega,\Cplx^N)}}$ it is sufficient to show that $\|D_0f\|_2=\|\nabla f\|_2$ for any $f\in C^\infty_c(\Omega, \Cplx^N)$. Let $f\in C^\infty_c(\Omega, \Cplx^N)$ be arbitrary. We have 
        \begin{align*}
            \langle D_0f, D_0f\rangle_{L^2(\Omega,\Cplx^N)} &= \int_\Omega \sum_{k,j=1}^d \langle \gamma_j\gamma_k\cdot (\partial_k f)(\bx), (\partial_j f)(\bx)\rangle_{\Cplx^N}\, d\bx\\
                                                            &= \sum_{k=1}^d\int_\Omega \langle (\partial_k f)(\bx), (\partial_k f)(\bx)\rangle_{\Cplx^N}\, d\bx\\
                                                            &\quad +\sum_{k\neq j}\int_\Omega \langle\gamma_j\gamma_k\cdot (\partial_k f)(\bx), (\partial_j f)(\bx)\rangle_{\Cplx^N}\, d\bx.
        \end{align*}
        The first term on the right-hand side above is $\|\nabla f\|_2^2$. Hence, it is sufficient to show that 
        $$\sum_{k\neq j}\int_\Omega \langle \gamma_j\gamma_k (\partial_k f)(\bx), (\partial_j f)(\bx)\rangle_{\Cplx^N}\, d\bx=0.$$

        Using the Leibniz rule, we have 
        \begin{align*}
            \sum_{k\neq j}\int_\Omega \langle \gamma_j\gamma_k (\partial_k f)(\bx), (\partial_j f)(\bx)\rangle_{\Cplx^N}\, d\bx
                    &=\sum_{k\neq j} \int_\Omega \partial_j\langle \gamma_j\gamma_k (\partial_k f)(\bx), f(\bx)\rangle_{\Cplx^N}\, d\bx\\
                    &\quad -\sum_{k\neq j} \int_\Omega \langle \gamma_j\gamma_k (\partial_j\partial_k f)(\bx), f(\bx)\rangle_{\Cplx^N}\, d\bx
        \end{align*}
        Since $f\in C^\infty_c(\Omega,\Cplx^N)$ and the matrices $\gamma_j$ and $\gamma_k$, $k\neq j$ anticommute, we have that 
        \begin{align*}
            \sum_{k\neq j} \langle \gamma_j\gamma_k (\partial_j\partial_k f)(\bx), f(\bx)\rangle_{\Cplx^N} &= \sum_{k<j}\langle \Big(\gamma_j\gamma_k (\partial_j\partial_k f)(\bx)-\gamma_j\gamma_k (\partial_k\partial_j f)(\bx)\Big), f(\bx)\rangle_{\Cplx^N}                                                                                \\
                                                                                                           &=0.
        \end{align*}
        Therefore, 
        \begin{align*}
            \sum_{k\neq j}\int_\Omega &\langle \gamma_j\gamma_k (\partial_k f)(\bx), (\partial_j f)(\bx)\rangle_{\Cplx^N}\, d\bx=\sum_{k\neq j} \int_\Omega \partial_j\langle \gamma_j\gamma_k (\partial_k f)(\bx), f(\bx)\rangle_{\Cplx^N}\, d\bx.
        \end{align*}
        Using again \eqref{Green div} for the latter integral, we obtain that:
        $$\sum_{k\neq j}\int_\Omega \langle \gamma_j\gamma_k (\partial_k f)(\bx), (\partial_j f)(\bx)\rangle_{\Cplx^N}\, d\bx=0,$$
        as required. 
    \end{proof} 
    
    
    Since $W^{1,2}_0(\Omega,\Cplx^N)$ is, by definition, the closure of $C^\infty_c(\Omega, \Cplx^N)$, Lemma \ref{lemma_D_0_symmetric} implies the following:   
    \begin{corollary}\label{cor_closure_D_0}
        Let $D_0=\Dd$ with $\dom(D_0)=C^\infty_c(\Omega,\Cplx^N)$ and let $D=\Dd$ with $\dom(D)=W^{1,2}_0(\Omega, \Cplx^N)$. The operator $D$ is the closure $\overline{D_0}$ of the operator $D_0$. In particular, $D$ is symmetric and the graph norm 
        ${\|\cdot\|_{\Gamma(D)}}$ is equivalent to the Sobolev norm ${\|\cdot\|_{W^{1,2}(\Omega,\Cplx^N)}}$. 
    \end{corollary}   
    
    The following description of the adjoint of the Dirac operator with Dirichlet boundary conditions may be seen as a consequence of the principle
    of elliptic regularity \cite[Chapter 1, Theorem 7.2]{shubin}. It is also possible to give an elementary proof following \cite{Schmidt-1995}.
    \begin{theorem}\label{dom D^* description}
        The adjoint $D^*$ of the Dirac operator $D$ with Dirichlet boundary conditions is described as:
        \begin{equation*}
            \dom(D^*) = \{u \in W^{1,2}_{\loc}(\Omega,\Cplx^N)\cap L_2(\Omega,\Cplx^N)\;:\; \Dd u \in L_2(\Omega,\Cplx^N)\}
        \end{equation*}
        and $D^*u = \Dd u$ for $u \in \dom(D^*)$.
    \end{theorem}
    
    It follows from Theorem \ref{dom D^* description} above that the operators $D^*D$ and $DD^*$ are described as follows:
    \begin{proposition}\label{prop_DD^*_and_D^*D_exact}
        Let $D$ be the Dirac operator with Dirichlet boundary conditions on $\Omega$. We have the following:
        \begin{enumerate}[{\rm (i)}]
            \item\label{D*D identification} The operator $D^*D$ is the Dirichlet Laplacian in the domain $\Omega$, that is
                \begin{align*}
                    \dom(D^*D)&=W^{1,2}_0(\Omega,\Cplx^N)\cap W^{2,2}(\Omega,\Cplx^N),\\
                    D^*Du&=-(1\otimes \Delta) u, \quad u\in \dom(D^*D).
                \end{align*} 

            \item\label{DD* identification} The operator $DD^*$ is defined as
                \begin{align*}
                    \dom(DD^*) &= \{u\in L_2(\Omega,\Cplx^N)\cap W^{1,2}_{\loc}(\Omega,\Cplx^N)\;:\; \Dd u \in W^{1,2}_0(\Omega,\Cplx^N)\},\\
                    DD^*u&=-(1\otimes \Delta) u, \quad u\in \dom(D^*D).
                \end{align*} 
        \end{enumerate}
    \end{proposition}
    \begin{proof}
        By the definition of $\dom(D^*D)$ we have that 
        \begin{align*}
            &\dom(D^*D)\\
            &=\Big\{h\in W_0^{1,2}(\Omega,\Cplx^N): D h\in{\rm dom}(D^*)\Big\}\\
            &=\Big\{h\in W_0^{1,2}(\Omega,\Cplx^N): \Dd h\in L_2(\Omega,\Cplx^N)\cap W^{1,2}_{\loc}(\Omega,\Cplx^N),\quad \Dd^2 h\in L_2(\Omega,\Cplx^N)\Big\}.
        \end{align*}

        We will now show that the operator $D^*D$ is exactly the well-known Dirichlet Laplacian (as defined in e.g. \cite[Theorem 10.19]{Schmudgen-2012}).
        Following \cite[Theorem 10.19]{Schmudgen-2012}, the Dirichlet Laplacian $-1\otimes \Delta_D$ on $L_2(\Omega,\Cplx^N)$ is defined on the domain
        $$\dom(-1\otimes \Delta_D)=W_0^{1,2}(\Omega,\Cplx^N)\cap W^{2,2}(\Omega,\Cplx^N)$$
        and that $-1\otimes \Delta_D$ acts as the distributional Laplacian on its domain. We now show that $D^*D$ is exactly $-1\otimes \Delta_D$.

        It is clear that $W_0^{1,2}(\Omega,\Cplx^N)\cap W^{2,2}(\Omega,\Cplx^N)\subset \dom(D^*D)$. Hence, since both the Dirichlet Laplacian and $D^*D$ acts as distributional Laplacian on their domain, we have that 
        $$-1\otimes \Delta_D\subset D^*D.$$
        Since $D$ is closed,  the operator $D^*D$ is self-adjoint (see e.g. \cite[Theorem X.25]{Reed-Simon-II-1975}). Hence, $D^*D$ is a self-adjoint extension of a self-adjoint operator $-1\otimes\Delta_D,$ and therefore, $D^*D=-1\otimes\Delta_D.$

        The second assertion of the proposition follows directly from the definition of the operator $DD^*$.
    \end{proof}
    
    \begin{remark}\label{mapping_remark}
	It follows from Proposition \ref{prop_DD^*_and_D^*D_exact} by induction that for all $k \geq 1$ we have:
	\begin{equation*}
	    \dom((DD^*)^k) \subseteq W^{2k,2}_{\loc}(\Omega,\Cplx^N).
	\end{equation*}
	Since,
	\begin{equation*}
	    W^{2k,2}_0(\Omega,\Cplx^N) \subseteq \dom(D^{2k})
	\end{equation*}
	it follows that for all $f \in C^\infty_c(\Omega)$ and $k\geq 0$, we have:
	\begin{equation*}
	    1\otimes M_f:\dom((DD^*)^k)\to \dom(D^{2k})
	\end{equation*}
	and hence,
	\begin{equation*}
	    1\otimes M_f:\dom((DD^*)^k)\to \dom((D^*D)^k).
	\end{equation*}
    \end{remark}
%

    Given the above description of $\dom(D^*D)$, we obtain the following:
    \begin{theorem}\label{dirichlet dimension}
        Let $f \in C^\infty_c(\Omega)$ and let $k$ be a positive integer. Then,
        \begin{equation*}
            (1\otimes M_f)(1+D^*D)^{-k} \in \Lc_{d/(2k),\infty}
        \end{equation*}
        and also:
        \begin{equation*}
            (1\otimes M_f)(1+DD^*)^{-k} \in \Lc_{d/(2k),\infty}.
        \end{equation*}
    \end{theorem}
    \begin{proof}
        Since $f$ is compactly supported, we may select an open set $U$ with smooth boundary such that $f$ is supported in $U$ and $U$ has compact closure in $\Omega$. Due
        to Remark \ref{mapping_remark}, we have that:
        \begin{equation*}
            (1\otimes M_f)(1+D^*D)^{-k}:L_2(\Omega,\Cplx^N)\to W^{2k,2}_0(U,\Cplx^N).
        \end{equation*}
        Using Remark \ref{mapping_remark} applied to $U$ instead of $\Omega$, the space $W^{2k,2}_0(\Omega,\Cplx^N)$ is contained in the domain of the $k$th power of the Dirichlet Laplace operator $\Delta_{D,U}$ on $U$, so the operator:
        \begin{equation*}
            (1-\Delta_{D,U})^k(1\otimes M_f)(1+D^*D)^{-k}
        \end{equation*}
        is closed and everywhere defined, and hence bounded. However since the closure of $U$ is compact and has smooth boundary, the operator $(1-\Delta_{D,U})^{-k}$
        is in $\Lc_{d/(2k),\infty}$ \cite[Theorem 12.14]{Schmudgen-2012}. Hence $(1\otimes M_f)(1+D^*D)^{-k} \in \Lc_{d/(2k),\infty}$.
        
        To prove the second claim, we use the fact that $1\otimes M_f$ maps $\dom((DD^*)^k)$ into $\dom((D^*D)^k)$. This implies that the operator:
        \begin{equation*}
            T := (1+D^*D)^{k}(1\otimes M_f)(1+DD^*)^{-k}
        \end{equation*}
        is everywhere defined and closed, and hence bounded by the closed graph theorem. Select $\phi \in C^\infty_c(\Omega)$ such that $\phi f = f$. Then,
        \begin{equation*}
            (1\otimes M_f)(1+DD^*)^{-k} = (1\otimes M_{\phi})(1+D^*D)^{-k}T.
        \end{equation*}
        Since $(1\otimes M_f)(1+D^*D)^{-k} \in \Lc_{d/(2k),\infty}$, we conclude that $(1\otimes M_f)(1+DD^*)^{-k} \in \Lc_{d/(2k),\infty}$.
    \end{proof}
    
    An application of the principle of elliptic regularity (as in \cite[Chapter 1, Theorem 7.2]{shubin}) yields the following:
    \begin{proposition}\label{helmholtz equation}
        The subspace of $u \in L_2(\Omega)$ which solve the Helmholtz equation:
        \begin{equation*}
            (1-\Delta)u = 0
        \end{equation*}
        is a closed subspace of $L_2(\Omega)$, and consists of smooth functions.
    \end{proposition}
    According to the above proposition, we may consider the projection $P:L_2(\Omega)\to L_2(\Omega)$ onto
    the closed subspace of $u$ such that $(1-\Delta)u = 0$.
    
    \begin{lemma}\label{image of T}
        Let $T$ be the linear operator:
        \begin{equation*}
            T := (1+D^*D)^{-1}-(1+DD^*)^{-1}.
        \end{equation*}
        Then $T = (1\otimes P)T$.
    \end{lemma}
    \begin{proof}
        Let $u \in L_2(\Omega,\Cplx^N)$, and define:
        \begin{equation*}
            v = (1+D^*D)^{-1}u,\quad w = (1+DD^*)^{-1}u.
        \end{equation*}
        It follows that $v$ and $w$ are in the domains of $D^*D$ and $DD^*$ respectively. Since both $D^*D$ and $DD^*$ act as $-1\otimes \Delta$ on their respective domains, we have:
        \begin{equation*}
            (1\otimes (1-\Delta))v = (1\otimes (1-\Delta))w = u.
        \end{equation*}
        Thus,
        \begin{equation*}
            (1\otimes (1-\Delta))(v-w) = 0.
        \end{equation*}
        Hence,
        \begin{equation*}
            (1\otimes P)(v-w) = v-w.
        \end{equation*}
        Since $v-w = Tu$, we have:
        \begin{equation*}
            (1\otimes P)Tu = Tu,\quad u \in L_2(\Omega,\Cplx^N)
        \end{equation*}
        and this completes the proof.
    \end{proof}
    
    \begin{corollary}\label{difference of resolvents}
        Let $f \in C^\infty_c(\Omega)$. Then,
        \begin{equation*}
            (1\otimes M_f)((1+D^*D)^{-1}-(1+DD^*)^{-1}) \in \bigcap_{s > 0} \Lc_s.
        \end{equation*}
    \end{corollary}
    \begin{proof}
        According to Lemma \ref{image of T}, if $T = (1+D^*D)^{-1}-(1+DD^*)^{-1}$ then:
        \begin{equation*}
            (1\otimes M_f)T = (1\otimes M_f)(1\otimes P)T.
        \end{equation*}
        Select $\phi \in C^\infty_c(\Omega)$ such that $\phi f = f$. Then,
        \begin{equation*}
            (1\otimes M_f)T = (1\otimes M_{\phi})(1\otimes M_f)(1\otimes P)T.
        \end{equation*}
        By Proposition \ref{helmholtz equation}, the image of $1\otimes P$ consists of smooth functions, and hence the image of $(1\otimes M_f)(1\otimes P)$
        is contained in $C^\infty_c(\Omega,\Cplx^N)$. This in particular is in the domain of $(1+D^*D)^k$, for all $k$, and therefore:
        \begin{equation*}
            (1\otimes M_f)T = (1\otimes M_{\phi})(1+D^*D)^{-k}(1+D^*D)^k(1\otimes M_f)(1\otimes P)T.
        \end{equation*}
        Since $D^*D$ is closed, the operator $(1+D^*D)^k(1\otimes M_f)(1\otimes P)$ is closed and everywhere defined, and hence bounded. Now by Theorem \ref{dirichlet dimension}, $(1\otimes M_{\phi})(1+D^*D)^{-k} \in \Lc_{d/(2k),\infty}$.
        But since $k>0$ is arbitrary, we have:
        \begin{equation*}
            (1\otimes M_f)T \in \bigcap_{k>0} \Lc_{d/(2k),\infty},
        \end{equation*}
        Thus,
        \begin{equation*}
            (1\otimes M_f)((1+D^*D)^{-1}-(1+DD^*)^{-1}) \in \bigcap_{s > 0} \Lc_{s}.
        \end{equation*} 
    \end{proof}

    \begin{remark}\label{rem_model_factorisation}
        We note that the algebra $C^\infty_c(\Omega)$ has the factorisation property, that is, for all $a\in C^\infty_c(\Omega)$ there exists $a_1,a_2\in C^\infty_c(\Omega)$, such that $a=a_1a_2.$
    \end{remark}

    Combining Theorem \ref{dirichlet dimension}, Corollary \ref{difference of resolvents} and Remark \ref{rem_model_factorisation}, 
    we see that the triple $(C^\infty_c(\Omega),L_2(\Omega,\Cplx^N),D)$ will satisfy the assumptions
    of Hypothesis \ref{lotoreichik hypothesis}, save smoothness. 
    
    By our definition of smoothness of a pre-spectral triple (Definition \ref{smoothness definition}), it is sufficient to show that the operators 
    $$R^k_{|D|}(a),\, R^k_{|D^*|}(a),$$
    and 
    $$R^k_{|D|}(\partial(a)),\, R^k_{|D^*|}(\partial(a)),$$
    extend to bounded operators for every $k\in\Ntrl$ and every $a\in C^\infty_c(\Omega)$. Moreover,
    \begin{equation*}
        \partial(a) = \sum_{j=1}^d -i\gamma_j \partial_j(a).
    \end{equation*}
    Thus since $|D^*|^2$ and $|D|^2$ commute with each $\gamma_j$ and since $\partial_j(a) \in C^\infty_c(\Omega)$, it suffices
    to show only the following:   
    
    \begin{lemma}\label{big smoothness proof}
        For all $f \in C^\infty_c(\Omega)$ and $k\geq 0$ we have that:
        \begin{equation*}
            R^k_{|D^*|}(1\otimes M_f),\,R^k_{|D|}(1\otimes M_f)
        \end{equation*}
        have bounded extension on $L_2(\Omega,\Cplx^N)$.
    \end{lemma}
    \begin{proof}
        We first prove the statement involving $D^*$.
        
        Since $DD^*$ acts as $\sum_{j=1}^d -1\otimes \partial_j^2$ on its domain, the operator $R_{|D^*|}(1\otimes M_f)$ is computed as:
        \begin{align*}
            [|D^*|^2,1\otimes M_f](1+|D^*|^2)^{-1/2} &= \sum_{j=1}^d -[\partial_j^2,1\otimes M_f](1+|D^*|^2)^{-1/2}\\
                                            &= \sum_{j=1}^d -2\partial_j(1\otimes M_{\partial_j f})(1+|D^*|^2)^{-1/2}\\
                                            &\quad +\sum_{j=1}^d(1\otimes M_{\partial^2_jf})(1+|D^*|^2)^{-1/2}
        \end{align*}
        Thus for the higher iterated commutators, we have by induction that $R^k_{|D^*|}(1\otimes M_f)$ is of the form:
        \begin{equation}\label{big expression for R^k}
            R^{k}_{|D^*|}(1\otimes M_f) = \sum_{l=0}^k \sum_{n_{1},n_2,\ldots,n_l} \partial_{n_1}\cdots\partial_{n_l}(1\otimes M_{\phi_{n_1,\ldots,n_l}})(1+|D^*|^2)^{-k/2}
        \end{equation}
        where the sum has a finite number of nonzero terms and $\phi_{n_1,\ldots,n_l}$ are some functions in $C^\infty_c(\Omega)$. Therefore it suffices to show that:
        \begin{equation}\label{partials then M_f then resolvent}
            \partial_{n_1}\cdots\partial_{n_l}(1\otimes M_{\phi})(1+|D^*|^2)^{-k/2}
        \end{equation}
        has bounded extension, for $\phi \in C^{\infty}_c(\Omega)$ and $l\leq k$. We note that this operator is defined everywhere, since $(1+|D^*|^2)^{-k/2}$
        maps $L_2(\Omega,\Cplx^N)$ to $\dom(|D^*|^k)$, and then $1\otimes M_{\phi}$ maps $\dom(|D^*|^k)$ to $\dom(D^k)$, and on this subspace the operator $\partial_{n_1}\cdots\partial_{n_l}$
        is defined.
        
        We can show that the operator in \eqref{partials then M_f then resolvent} has bounded extension by considering $1\otimes M_{\phi}$ as a map from $\dom(|D^*|^k)$ to $W^{k,2}(\Rl^d,\Cplx^N)$, as follows: let $\iota$ denote the isometric injection of $W^{k,2}_0(\Omega)$ into $W^{k,2}(\Rl^d)$, and 
        let $\rho$ be the ``restriction to $\Omega$" map $W^{k,2}(\Rl^d)\to W^{k,2}(\Omega)$. 
        Then,
        \begin{equation*}
            \partial_{n_1}\cdots\partial_{n_l}(1\otimes M_{\phi})(1+|D^*|^2)^{-k/2} = \rho\partial_{n_1}\ldots\partial_{n_l}\iota (1\otimes M_\phi)(1+|D^*|^2)^{-k/2}.
        \end{equation*}
        This is possible because $\phi$ is compactly supported in $\Omega$, and so for $\xi \in \dom(|D^*|^k) \subseteq W^{k,2}(\Omega,\Cplx^N)$ we have that $M_{\phi}\xi \in W^{k,2}_0(\Omega,\Cplx^N)$. Now inserting a factor of $(1-\Delta_{\Rl^d})^{k/2}$, where $\Delta_{\Rl^d}$ is the Laplace operator on $\Rl^d$, we have:
        \begin{align*}
            \partial_{n_1}&\cdots\partial_{n_l}(1\otimes M_{\phi})(1+|D^*|^2)^{-k/2}\\
                          &= \rho\partial_{n_1}\cdots\partial_{n_k}(1-\Delta_{\Rl^d})^{-k/2}(1-\Delta_{\Rl^d})^{k/2}\iota(1\otimes M_{\phi})(1+|D^*|^2)^{-k/2}.
        \end{align*}
        By functional calculus, the operator $\partial_{n_1}\cdots \partial_{n_l}(1-\Delta_{\Rl^d})^{-k/2}$ is bounded on $L_2(\Rl^d)$, and so we will show that
        \begin{equation}\label{Omega to R^d}
            (1-\Delta_{\Rl^d})^{k/2}\iota(1\otimes M_{\phi})(1+|D^*|^2)^{-k/2}
        \end{equation}
        is bounded. To see this, we first note that the operator in \eqref{Omega to R^d} is everywhere defined, since $M_{\phi}$ maps $\dom(|D^*|^k)$
        to $W^{k,2}_0(\Omega)$, and $(1-\Delta_{\Rl^d})^{k/2}\iota$ is well-defined on this subspace. Next, the operator in \eqref{Omega to R^d} is closed, since $(1-\Delta_{\Rl^d})^{k/2}$ is closed and $M_{\phi}(1+|D^*|^2)^{-k/2}$ is bounded.
        Hence, by the closed graph theorem, the operator \ref{Omega to R^d} is bounded. This proves that each summand in \eqref{big expression for R^k} has bounded extension, and thus $R^k_{|D^*|}(1\otimes M_f)$ has bounded extension.
        
        Finally, to show that $R^k_{|D|}(1\otimes M_f)$ has bounded extension, we may use an identical argument since $M_{\phi}$ may be considered as mapping $\dom(|D|^n)$
        to $W^{k,2}(\Rl^d,\Cplx^N)$.
    \end{proof} 
    By combining the results above, we have verified that our model example indeed provides
    \pre-spectral triples satisfying Hypothesis \ref{lotoreichik hypothesis}.
    \begin{theorem}\label{examples_gives_pre_spectral_triple}
        Let $\Omega \subseteq \Rl^d$ be an arbitrary open set, and let $D$ be the Dirac operator with Dirichlet boundary conditions on $\Omega$, and let $H = L_2(\Omega,\Cplx^N)$,
        and let $\Ac = C^\infty_c(\Omega)$ act on $H$ as $1\otimes M_f, f\in C^\infty_c(\Omega)$.
        
        Then $(\Ac,H,D)$ is a smoothly $d$-dimensional \pre-spectral triple satisfying Hypothesis \ref{lotoreichik hypothesis}.
    \end{theorem}
    \begin{proof}        
        As has already been discussed, we have that $a \in \Ac$ maps $\dom(D^*)$ to $\dom(D)$.
        Due to Theorem \ref{dirichlet dimension}, we have:
        \begin{equation*}
            (1\otimes M_f)(1+D^*D)^{-1} \in \Lc_{d/2,\infty}.
        \end{equation*}
        and
        \begin{equation*}
            (1\otimes M_f)(1+DD^*)^{-1} \in \Lc_{d/2,\infty},
        \end{equation*}
        
        By an application of the Araki-Lieb-Thirring inequality \eqref{ALT inequality}, it follows that:
        \begin{equation*}
            (1\otimes M_f)(1+D^*D)^{-1/2},\,(1\otimes M_f)(1+DD^*)^{-1/2} \in \Lc_{d,\infty}.
        \end{equation*}
        Since for any $f \in C^\infty_c(\Omega)$, we can find $\phi \in C^\infty_c(\Omega)$ such that $[D,1\otimes M_f] = [D,1\otimes M_f](1\otimes M_{\phi})$, 
        it follows that $(C^\infty_c(\Omega),L_2(\Omega,\Cplx^N),D)$ is a $p$-dimensional \pre-spectral triple. Smoothness of the \pre-spectral triple
        is an immediate consequence of Lemma \ref{big smoothness proof}.
        
        Moreover, for any $f \in C^\infty_c(\Omega)$ we can find $\phi \in C^\infty_c(\Omega)$ such that $R^k_{|D^*|}(1\otimes M_f) = (1\otimes M_{\phi})R^k_{|D^*|}(1\otimes M_f)$
        and $R^k_{|D|}(1\otimes M_f) = (1\otimes M_\phi)R^k_{|D|}(1\otimes M_f)$. Hence, for all $f \in C^\infty_c(\Omega)$ and $k \geq 0$ we have:
        \begin{equation*}
            (1+D^*D)^{-1}R^k_{|D|}(1\otimes M_f),\, (1+DD^*)^{-1}R^k_{|D^*|}(1\otimes M_f) \in \Lc_{d/2,\infty}
        \end{equation*}
        Then commuting $(1+D^*D)^{-1}$ with $R^k_{|D|}(1\otimes M_f)$ yields:
        \begin{align*}
            (1+D^*D)^{-1}R^k_{|D|}(1\otimes M_f) &= -(1+D^*D)^{-1}R^{k+1}_{|D|}(1\otimes M_f)(1+D^*D)^{-1}\\
                                                 &\quad +R^k_{|D|}(1\otimes M_f)(1+D^*D)^{-1}\\
                                                 &= -(1+D^*D)^{-1}M_{\phi}R^{k+1}_{|D|}(1\otimes M_f)(1+D^*D)^{-1}\\
                                                 &\quad +R^k_{|D|}(1\otimes M_f)(1+D^*D)^{-1}.
        \end{align*}
        Thus,
        \begin{equation*}
            R^{k}_{|D|}(1\otimes M_f)(1+D^*D)^{-1} \in \Lc_{d/2,\infty}
        \end{equation*}
        and applying the Araki-Lieb-Thirring inequality:
        \begin{equation*}
            R^k_{|D|}(1\otimes M_f)(1+D^*D)^{-1/2} \in \Lc_{d,\infty}.
        \end{equation*}
        By an identical proof:
        \begin{equation*}
            R^k_{|D^*|}(1\otimes M_f)(1+DD^*)^{-1/2} \in \Lc_{d,\infty}.
        \end{equation*}        
        an argument similar to that of Corollary \ref{sufficient conditions for double to satisfy 1.2.1} shows that the \double satisfies
        Hypothesis \ref{main assumption}, and hence $(C^\infty_c(\Omega),L_2(\Omega,\Cplx^N),D)$ is smoothly $p$-dimensional.
        
        Finally, the remainder of Hypothesis \ref{lotoreichik hypothesis} follows immediately from Corollary \ref{difference of resolvents},
    \end{proof}
    Since all of the conditions of Theorem \ref{character theorem} are verified, it follows that our version of the Character theorem
    holds for the model example $(C^\infty_c(\Omega),L_2(\Omega,\Cplx^N),D)$.
    
\section*{Appendices}


\renewcommand{\thesection}{A}
    
\section{Proof of Lemma \ref{lambda to delta}}\label{lambda to delta appendix}
    Here we include the proof of Lemma \ref{lambda to delta}. For this appendix $(\Bc,K,T)$ is a spectral triple, $x \in \Bc \cup \partial_0(\Bc)$ and $p \geq 1$.
    
    We recall the notation $|T_0| = (1+T^2)^{1/2}$ and $\delta_0(x) = [|T_0|,x]$.
    One direction of the equivalence is not difficult:
    \begin{lemma}
        If for all $k\geq 0$ we have:
        \begin{equation*}
            \delta_0^k(x)(1+T^2)^{-1/2} \in \Lc_{p,\infty}
        \end{equation*}
        then, for all $k\geq 0$,
        \begin{equation*}
            R_T^k(x)(1+T^2)^{-1/2} \in \Lc_{p,\infty}
        \end{equation*}
    \end{lemma}
    \begin{proof}
        By the Leibniz rule,
        \begin{align*}
                  R_T(x) &= [|T_0|^2,x]|T_0|^{-1}\\
                         &= [|T_0|,x]+|T_0|[|T_0|,x]|T_0|^{-1}\\
                         &= 2\delta_0(x)+\delta^2(x)|T_0|^{-1}.
        \end{align*}
        Then by the binomial theorem,
        \begin{equation*}
            R_T^k(x) = \sum_{l=0}^k \binom{k}{l} 2^l\delta_0^{2k-l}(x)|T_0|^{-k+l}.
        \end{equation*}
        So multiplying on the right by $|T_0|^{-1}$,
        \begin{equation*}
            R_T^k(x)|T_0|^{-1} = \sum_{l=0}^k \binom{k}{l}2^l\delta_0^{2k-l}(x)|T_0|^{-k+l-1}
        \end{equation*}
        Since by assumption each summand is in $\Lc_{p,\infty}$, it follows that $R_T^k(x)$ is in $\Lc_{p,\infty}$.     
    \end{proof}
    
    So now we focus on proving the reverse implication in Lemma \ref{lambda to delta}. The main difficulty that we have is with the $p=1$ case due to the lack of a Banach norm on $\Lc_{1,\infty}$. Instead,
    we use recently developed operator integration techniques from \cite{LeSZ2}. The following is (a special case of) \cite[Corollary 3.7]{LeSZ2}: Recall that for a closed bounded interval $I$ and a quasi-Banach space $X$, the space
    $C^{2}(I,X)$ is the set of twice continuously differentiable functions $F:I\to X$, equipped with the quasi-norm:
    \begin{equation*}
        \|F\|_{C^2(I,X)} = \max\{\sup_{t \in I} \|F(t)\|_X,\,\sup_{t \in I} \|F'(t)\|_X,\,\sup_{t \in I} \|F''(t)\|_X\}.
    \end{equation*}
    Let $p \in (1/2,\infty)$ and let $F \in C^2([0,\infty),\Lc_{p,\infty})$. If
    \begin{equation*}
        \sum_{j=0}^\infty \|F\|_{C^2([j,j+1],\Lc_{p,\infty})}^{\frac{p}{p+1}} < \infty
    \end{equation*}
    then $\int_0^\infty F(\lambda)\,d\lambda \in \Lc_{p,\infty}$, and there is a constant $c_p$ so that we have a quasi-norm bound:
    \begin{equation*}
        \left\|\int_0^\infty F(\lambda)\,d\lambda\right\|_{p,\infty} \leq c_p\left(\sum_{j=0}^\infty \|F\|_{C^2([j,j+1],\Lc_{p,\infty})}^{\frac{p}{p+1}}\right)^{\frac{p+1}{p}}.
    \end{equation*}
    
    \begin{lemma}\label{lambda to delta auxiliary lemma}
        Let $X \in \Lc_\infty(K)$ be a linear operator such that $R_T(X)$, $R^2_T(X)$, $R^3_T(X)$ and $R^4_T(X)$ are in  $\Lc_{p,\infty}$, with $p \geq 1$. Then
        $\delta_0(X) \in \Lc_{p,\infty}$.
    \end{lemma}
    \begin{proof}
        Recall that $T_0^2 = 1+T^2$. Starting from the integral formula:
        \begin{equation*}
            (1+T^2)^{1/2} = \frac{1}{\pi}\int_0^\infty \frac{1+T^2}{1+\lambda+T^2}\lambda^{-1/2}\,d\lambda
        \end{equation*}
        (see e.g. \cite[Remark 3]{CP})
        we have:{\highlight
        \begin{align*}
            \delta_0(X) &= \frac{1}{\pi} \int_0^\infty \left[\frac{T_0^2}{\lambda+T_0^2},X\right]\lambda^{-1/2}\,d\lambda\\
                        &= \frac{1}{\pi} \int_0^\infty \left[1-\frac{\lambda}{\lambda+T_0^2},X\right]\lambda^{-1/2}\,d\lambda\\
                        &= \frac{1}{\pi} \int_0^\infty \lambda^{1/2}\frac{1}{\lambda+T_0^2}[T_0^2,X]\frac{1}{\lambda+T_0^2}\,d\lambda\\
                        &= \frac{1}{\pi} \int_0^\infty \lambda^{1/2}\frac{1}{\lambda+T_0^2}R_T(X)\frac{T_0}{\lambda+T_0^2}\,d\lambda\\
                        &= \frac{1}{\pi} \int_0^\infty -\lambda^{1/2}\frac{1}{\lambda+T_0^2}R^2_T(X)\frac{T_0^2}{(\lambda+T_0^2)^2} + \lambda^{1/2}R_T(X)\frac{T_0}{(\lambda+T_0^2)^2}\,d\lambda\\
                        &= \frac{1}{2}R_T(X)-\frac{1}{\pi}\int_0^\infty \lambda^{1/2}\frac{1}{\lambda+T_0^2}R^2_T(X)\frac{T_0^2}{(\lambda+T_0^2)^2}\,d\lambda\\
                        &= \frac{1}{2}R_T(X)-\frac{1}{8}R^2_T(X)T_0^{-1}+\frac{1}{\pi}\int_0^{\infty} \lambda^{1/2}\frac{1}{\lambda+T_0^2}R^3_T(X)\frac{T_0^3}{(\lambda+T_0^2)^3}\,d\lambda\\
                        &= \frac{1}{2}R_T(X)-\frac{1}{8}R^2_T(X)T_0^{-1}+\frac{1}{16}R^3_T(X)T_0^{-2}-\frac{1}{\pi}\int_0^\infty \lambda^{1/2}\frac{1}{\lambda+T_0^2}R^4_T(X)\frac{T_0^4}{(\lambda+T_0^2)^4}\,d\lambda.
        \end{align*}
        We now examine the latter integral with a view to applying \cite[Corollary 3.7]{LeSZ2} in order to show that it is in $\Lc_{p,\infty}$. Define
        \begin{equation*}
            F(\lambda) = A(\lambda)R^4_T(X)B(\lambda),\quad \lambda \in [0,\infty)
        \end{equation*}
        where
        \begin{align*}
            A(\lambda) &:= \frac{\lambda^{1/2}}{\lambda+T_0^2}\\
            B(\lambda) &:= \frac{T_0^4}{(\lambda+T_0^2)^4}.
        \end{align*}
        We will use the estimate (see e.g. \cite[Proposition 3.8]{LeSZ2}):
        \begin{equation*}
            \|F\|_{C^2([j,j+1],\Lc_{p,\infty})} \leq \|A\|_{C^2([j,j+1],\Lc_\infty)}\|R^2_T(X)\|_{\Lc_{p,\infty}}\|B\|_{C^2([j,j+1],\Lc_\infty)}.
        \end{equation*}
        Let us estimate the $C^2$ norms of $A$ and $B$. First, for $A$ we have:
        \begin{align*}
            A'(\lambda)  &= \frac{1}{2\lambda^{1/2}(\lambda+T_0^2)}-\frac{\lambda^{1/2}}{(\lambda+T_0^2)^2}\\
            A''(\lambda) &= -\frac{1}{4\lambda^{3/2}(\lambda+T_0^2)}-\frac{1}{\lambda^{1/2}(\lambda+T_0^2)^2}+2\frac{\lambda^{1/2}}{(\lambda+T_0^2)^3}
        \end{align*}
        and for $B$,
        \begin{align*}
            B'(\lambda)  &= -\frac{4T_0^4}{(\lambda+T_0^2)^5}\\
            B''(\lambda) &= \frac{20T_0^4}{(\lambda+T_0^2)^6}.
        \end{align*}
        For fixed $\lambda \in (0,\infty)$, we then have:
        \begin{equation*}
            \|A(\lambda)\|_{\infty} \leq \lambda^{-1/2},\;\|A'(\lambda)\|_\infty  \leq C_1\lambda^{-3/2},\;\|A''(\lambda)\|_\infty \leq C_2\lambda^{-5/2}
        \end{equation*}
        and
        \begin{equation*}
            \|B(\lambda)\|_\infty \leq \lambda^{-2},\; \|B'(\lambda)\|_\infty \leq C_3\lambda^{-3},\; \|B''(\lambda)\|_\infty \leq C_4\lambda^{-4}.
        \end{equation*}
        
        So for $j\geq 0$,
        \begin{align*}
            \|A\|_{C^2([j,j+1],\Lc_\infty)} &\leq C(j+1)^{-1/2},\\
            \|B\|_{C^2([j,j+1],\Lc_\infty)} &\leq C(j+1)^{-2}.
        \end{align*}}
        Thus,
        \begin{equation*}
            \|F\|_{C^2([j,j+1],\Lc_{p,\infty})} \leq C(j+1)^{-5/2}.
        \end{equation*}
        Then applying \cite[Corollary 3.7]{LeSZ2}, it follows that $\int_0^\infty F(\lambda)\,d\lambda \in \Lc_{p,\infty}$ if:
        \begin{equation*}
            \sum_{j\geq 0} (j+1)^{-\frac{5p}{2p+2}} < \infty.
        \end{equation*}
        This is indeed the case if $3p > 2$, which holds due to our assumption that $p \geq 1$. Thus,
        \begin{equation*}
            \int_0^\infty F(\lambda)\,d\lambda \in \Lc_{p,\infty}
        \end{equation*}
        and so $\delta_0(X) \in \Lc_{p,\infty}$, as required.
    \end{proof}
    
    \begin{proof}[Proof of Lemma \ref{lambda to delta}]
        Assume that $R_T^k(x)(1+T^2)^{-1/2} \in \Lc_{p,\infty}$ for all $k \geq 0$.
        
        We will show that:
        \begin{equation}\label{square to prove}
            \delta_0^{j}(R_T^{k}(x))(1+T^2)^{-1/2} \in \Lc_{p,\infty}
        \end{equation}
        for all $(j,k) \in \Ntrl^2$. Let $\Zc$ be the set of $(j,k)$ such that \eqref{square to prove} holds. By assumption, $(0,k) \in \Zc$ for all $k\geq 0$.
        
        By Lemma \ref{lambda to delta auxiliary lemma} with $X= \delta_0^j(R_T^k(x))(1+T^2)^{-1/2}$, we have:
        \begin{equation}\label{inductive argument on the square}
            (j,k+1),\,(j,k+2) \in \Zc\Rightarrow (j+1,k) \in \Zc.
        \end{equation}
        
        It then follows that $\Zc = \Ntrl^2$. To see this, suppose that there is $(j,k) \in \Ntrl^2\setminus \Zc$, and choose $(j,k)$ such that $j$ is minimal. Since $(0,k) \in \Zc$ for all $k\geq 0$,
        we must have $j\geq 1$. However we must have $(j-1,k+1),(j-1,k+2) \in \Zc$ since we assume that $j$ is minimal. Then by \eqref{inductive argument on the square} it follows
        that $(j,k) \in \Zc$.
    \end{proof}

\renewcommand{\thesection}{B}
    
\section{Proof of Theorem \ref{difference is trace class}}\label{operator difference appendix}
    The following result concerns conditions on positive bounded operators $A$ and $B$ so that the difference $B^rA^r-(A^{1/2}BA^{1/2})^r$ is trace class.
    The proof is based on two operator integration results: the first is the following integral representation which appeared first as a special case in \cite[Lemma 5.2]{CSZ}, and was later strengthened in \cite{SZ-asterisque}.
    
    The following appears as \cite[Theorem 5.2.1]{SZ-asterisque}: Let $A$ and $B$ be positive bounded operators on a complex separable Hilbert space $H$, let $z \in \Cplx$ with $\Re(z) > 1$ and let $Y = A^{1/2}BA^{1/2}$. We define the mapping $T_z:\Rl\to \Lc_\infty$ by:
    \begin{align*}
        T_z(0) &:= B^{z-1}[BA^{1/2},A^{z-1/2}]+[BA^{1/2},A^{1/2}]Y^{z-1},\\
        T_z(s) &:= B^{z-1+is}[BA^{1/2},A^{z-1/2+is}]Y^{-is}+B^{is}[BA^{1/2},A^{1/2+is}]Y^{z-1-is},\quad s \neq 0.
    \end{align*}
    Here, we use the convention that $0^{is} = 0$ for all $s \in \Rl$. We consider the function $g_z:\Rl\to \Cplx$ given by:
    \begin{align*}
        g_z(0) &:= 1-\frac{z}{2},\\
        g_z(t) &:= 1-\frac{e^{\frac{z}{2}t}-e^{-\frac{z}{2}t}}{(e^{\frac{t}{2}}-e^{-\frac{t}{2}})(e^{\frac{z-1}{2}t}-e^{-\frac{z-1}{2}t})},\quad t \neq 0.
    \end{align*}
    Note that $g_z$ is a Schwartz function (See \cite[Remark 5.22]{SZ-asterisque}).
    Then $T_z:\Rl\to\Lc_\infty$ is continuous in the weak operator topology, and if 
    $$\widehat{g}_z(s) := (2\pi)^{-1}\int_{-\infty}^\infty g_z(t)e^{-ist}\,dt,$$
    then:
    \begin{equation}\label{integral representation}
        B^zA^z-(A^{1/2}BA^{1/2})^z = T_z(0)-\int_{\Rl} T_z(s)\widehat{g}_z(s)\,ds.
    \end{equation}
    
    The second technical result we use is that if $r > 1$ and $X$ and $Y$ are positive operators such that $[X,Y] \in \Lc_{r,1}$, and $f$ is a Lipschitz function on $\Rl$ then
    \begin{equation*}
        \|[X,f(Y)]\|_{r,1} \leq c_r\|f'\|_{L_\infty(\Rl)}\|[X,Y]\|_{r,1}.
    \end{equation*}
    This follows from \cite[Equation (14)]{PS-acta}, since $\Lc_{r,1}$ is an interpolation space between $\Lc_p$ and $\Lc_q$ spaces for some $1 < p < q <\infty$. In particular, if $f(t) = t^{1+2is}$ then $f$ is a Lipschitz function on $\Rl$, with $|f'(t)| = |1+2is|$ for $t \in \Rl$,
    and therefore:
    \begin{equation}\label{lipschitz bound}
        \|[X,Y^{1+2is}]\|_{r,1} \leq 2c_r(1+|s|)\|[X,Y]\|_{r,1}.
    \end{equation}
    
    \begin{theorem}\label{operator difference results}
        Let $A$ and $B$ be two positive bounded operators, and let $r > 1$. If the following four conditions hold:
        \begin{enumerate}[{\rm (i)}]
            \item{}\label{op diff 1} $B^{r-1}A^{r-1} \in \Lc_{\frac{r}{r-1},\infty}$,
            \item{}\label{op diff 2} $A^{1/2}BA^{1/2} \in \Lc_{r,\infty}$,
            \item{}\label{op diff 3} $[BA^{1/2},A^{1/2}] \in \Lc_{r,1}$,
            \item{}\label{op diff 4} $B^{r-1}[B,A^{r-1}]A \in \Lc_1$.
        \end{enumerate}
        Then
        \begin{equation*}
            B^rA^r-(A^{1/2}BA^{1/2})^r \in \Lc_1.
        \end{equation*}
    \end{theorem}
    \begin{proof}
        Taking $z = r$ and observing that $\Re(z) = r > 1$ allows us to apply \eqref{integral representation} to get:
        \begin{equation*}
            B^rA^r-(A^{1/2}BA^{1/2})^r = T_r(0)-\int_{\Rl} T_r(s)\widehat{g}_r(s)\,ds.
        \end{equation*}
        We now focus on proving that $T_r(0) \in \Lc_1$ and
        \begin{equation*}
            \int_{\Rl} \|T_r(s)\|_1|\widehat{g}_r(s)|\,ds < \infty.
        \end{equation*}
        
        Let $s \in \Rl$. As the function $t\mapsto t^{is}$ takes values in $\{z\in \Cplx\;:\;|z|=0,1\}$, the operator $Y^{is} = (A^{1/2}BA^{1/2})^{is}$ is a partial isometry. So we have by the triangle inequality:
        \begin{equation*}
            \|T_r(s)\|_1 \leq \|B^{r-1}[BA^{1/2},A^{r-1/2}]\|_1+\|[BA^{1/2},A^{1/2+is}]Y^{r-1}\|_1.
        \end{equation*}
        Note that this holds even in the $s=0$ case.
        By the Leibniz rule:
        \begin{align*}
            [BA^{1/2},A^{r-\frac{1}{2}+is}] &= [BA^{1/2},A^{r-1}A^{1/2+is}]\\
                                            &= [BA^{1/2},A^{r-1}]A^{1/2+is}+A^{r-1}[BA^{1/2},A^{1/2+is}].
        \end{align*}
        Therefore,
        \begin{align*}
            \|T_r(s)\|_1 &\leq \|B^{r-1}[B,A^{r-1}]A\|_1+\|B^{r-1}A^{r-1}[BA^{1/2},A^{1/2+is}]\|_1\\
                         &\quad +\|[BA^{1/2},A^{1/2+is}]Y^{r-1}\|_1.
        \end{align*}
        Using the H\"older inequality we have:
        \begin{align*}
            \|T_r(s)\|_1 &\leq \|B^{r-1}[B,A^{r-1}]A\|_1 +\|B^{r-1}A^{r-1}\|_{\frac{r}{r-1},\infty}\|[BA^{1/2},A^{1/2+is}]\|_{r,1}\\
                         &\quad +\|[BA^{1/2},A^{1/2+is}]\|_{r,1}\|Y^{r-1}\|_{\frac{r}{r-1},\infty}.
        \end{align*}
        By assumption \eqref{op diff 4}, the first norm $\|B^{r-1}[B,A^{r-1}]A\|_1$ is finite, and by \eqref{op diff 1} the norm $\|B^{r-1}A^{r-1}\|_{\frac{r}{r-1},\infty}$ is finite.
        Finally by \eqref{op diff 2}, we have $Y^{r-1} \in \Lc_{\frac{r}{r-1},\infty}$ and so $\|Y^{r-1}\|_{\frac{r}{r-1},\infty}$ is finite.
        
        So there are constants $c_1$ and $c_2$ such that:
        \begin{equation*}
            \|T_r(s)\|_1 \leq c_1+c_2\|[BA^{1/2},A^{1/2+is}]\|_{r,1}.
        \end{equation*}
        If $s = 0$, then by assumption \eqref{op diff 3} the latter norm is finite, so we have proved that $\|T_r(0)\|_1 < \infty$.
        
        Since by \eqref{op diff 3} we have that $[BA^{1/2},A^{1/2}]\in \Lc_{r,1}$, we can apply \eqref{lipschitz bound} with $X = BA^{1/2}$ and $Y = A^{1/2}$ to get:
        \begin{equation*}
            \|T_r(s)\|_1 \leq c_1+2c_2c_r(1+|s|).
        \end{equation*}
        Since $g_r$ is Schwartz, the Fourier transform $\widehat{g}_r$ is Schwartz, and therefore,
        \begin{equation*}
            \int_{\Rl} |\widehat{g}_r(s)|\|T_r(s)\|_1\,ds \leq c_1\int_{\Rl}|\widehat{g}_r(s)|\,ds+2c_2c_r\int_{\Rl}|\widehat{g}_r(s)|(1+|s|)\,ds\\
                                                          <\infty.
        \end{equation*}
        Thus $\int_\Rl \widehat{g}_r(s)T_r(s)\,ds \in \Lc_1$, and so $B^rA^r-(A^{1/2}BA^{1/2})^r \in \Lc_1$.       
    \end{proof}
    
    The following theorem is a recent result to Eric Ricard, and is a special case of \cite[Theorem 3.2]{Ricard-fractional-powers}:
    \begin{theorem}\label{ricard theorem}
        Let $s > 0$, and $\theta \in (0,1]$. If $X$ and $Y$ are positive operators such that $X-Y \in \Lc_s$, then $X^\theta-Y^\theta \in \Lc_{s/\theta}$, with concrete quasi-norm bound:
        \begin{equation*}
            \|X^\theta-Y^\theta\|_{\Lc_{s/\theta}} \leq C_{s,\theta}\|X-Y\|_{\Lc_s}^\theta.
        \end{equation*}
    \end{theorem}    

    \begin{proof}[Proof of Theorem \ref{difference is trace class}]
        The $p=1$ case follows immediately from Theorem \ref{ricard theorem} with $X = (1+D^*D)^{-1}$, $Y = (1+DD^*)^{-1}$ and $\theta = \frac{1}{2}$, and so we focus on the $p > 1$ cases,
    
        By the assumption of Hypothesis \ref{lotoreichik hypothesis}, for $0 \leq a \in \Ac$ we have:
        \begin{equation}\label{initial resolvent difference}
            a(1+D^*D)^{-1}a-a(1+DD^*)^{-1}a \in \Lc_{\frac{p}{3}}.
        \end{equation}
        Initially we handle the case $p=2$. Since $\Lc_{\frac{2}{3}}\subset \Lc_1$, we have:
        \begin{equation*}
            a(1+D^*D)^{-1}a-a(1+DD^*)^{-1}a \in \Lc_1.
        \end{equation*}
        We have included in Hypothesis \ref{lotoreichik hypothesis} that $a$ can be factorised as $a_1a_2$, for $a_1,a_2 \in \Ac$. Hence using the Leibniz rule and the simple observation that $[|D|^2,a]^* = -[|D|^2,a^*]$, we have:
        \begin{align*}
            [(1+D^*D)^{-1},a] &= [(1+D^*D)^{-1},a_1]a_2+a_1[(1+D^*D)^{-1},a_2]\\
                              &= -(1+|D|^2)^{-1}[|D|^2,a_1](1+|D|^2)^{-1}a_2\\
                              &\quad -a_1(1+|D|^2)^{-1}[|D|^2,a_2](1+|D|^2)^{-1}\\
                              &= ([|D|^2,a_1^*](1+|D|^2)^{-1})^*(1+|D|^2)^{-1}a_2\\
                              &\quad -a_1(1+|D|^2)^{-1}[|D|^2,a_2](1+|D|^2)^{-1}\\
                              &= (R_{|D|}(a_1^*)(1+|D|^2)^{-1/2})^*(1+|D|^2)^{-1}a_2\\
                              &\quad -a_1(1+|D|^2)^{-1}R_{|D|}(a_2)(1+|D|^2)^{-1/2}
        \end{align*}
        By the definition of being smoothly $2$-dimensional (Definition \ref{def_smoothly_p_dim}) and Lemma \ref{strong implies weak dimension} we have that for all $a \in \Ac$,
        \begin{equation*}
            a(1+|D|^2)^{-1} \in \Lc_{1,\infty},\quad R_{|D|}(a)(1+|D|^2)^{-1/2} \in \Lc_{2,\infty}.
        \end{equation*}
        Thus by H\"older's inequality:
        \begin{equation*}
            [(1+D^*D)^{-1},a] \in \Lc_{2,\infty}\cdot \Lc_{1,\infty}+\Lc_{1,\infty}\cdot \Lc_{2,\infty} \subseteq \Lc_{\frac{2}{3},\infty} \subset \Lc_1.
        \end{equation*}
        By an identical argument, we also have:
        \begin{equation*}
            [(1+DD^*)^{-1},a] \in \Lc_1.
        \end{equation*}
        Thus by \eqref{initial resolvent difference}:
        \begin{align*}
            (1+D^*D)^{-1}a^2-(1+DD^*)^{-1}a^2 &= [(1+D^*D)^{-1},a]a-[(1+DD^*)^{-1},a]a\\
                                              &\quad+a(1+D^*D)^{-1}a-a(1+DD^*)^{-1}a\\
                                              &\in \Lc_1
        \end{align*}        
        and this completes the proof for $p=2$.
        
    
        Now assume that $p > 2$. Applying Theorem \ref{ricard theorem} with $\theta = \frac{1}{2}$ to \eqref{initial resolvent difference} yields:
        \begin{equation}\label{difference of square roots}
            (a(1+D^*D)^{-1}a)^{\frac{1}{2}}-(a(1+DD^*)^{-1}a)^{\frac{1}{2}} \in \Lc_{2p/3} \subset \Lc_{2p/3,\infty}.
        \end{equation}
        By the assumption of Hypothesis \ref{lotoreichik hypothesis}, $p$ is an integer. Then we have that:
        \begin{align*}
            (a(1+D^*D)^{-1}a)^{\frac{p}{2}}-&(a(1+DD^*)^{-1}a)^{\frac{p}{2}} \\
              &= \sum_{k=0}^{p-1} (a(1+D^*D)^{-1}a)^{\frac{k}{2}}((a(1+D^*D)^{-1}a)^{\frac{1}{2}}\\
              &\quad -(a(1+DD^*)^{-1}a)^{\frac{1}{2}})(a(1+DD^*)^{-1}a)^{\frac{p-1-k}{2}}.
        \end{align*}
        Examining the $k$th summand, by Lemma \ref{strong implies weak dimension}, we have that $(a(1+DD^*)^{-1}a)^{k/2}$ is in $\Lc_{\frac{p}{k},\infty}$
        and $(a(1+D^*D)^{-1}a)^{\frac{p-1-k}{2}} \in \Lc_{\frac{p}{p-1-k},\infty}$. Thus by \eqref{difference of square roots} and the H\"older inequality,
        \begin{align*}
            (a(1+D^*D)^{-1}a)^{\frac{k}{2}}((a(1+D^*D)^{-1}&a)^{\frac{1}{2}}-(a(1+DD^*)^{-1}a)^{\frac{1}{2}})(a(1+DD^*)^{-1}a)^{\frac{p-1-k}{2}}\\
                                                           &\in \Lc_{\frac{p}{k},\infty}\cdot\Lc_{\frac{2p}{3},\infty}\cdot\Lc_{\frac{p}{p-k-1},\infty}\\
                                                           &= \Lc_{\frac{2p}{2p+1},\infty}\\
                                                           &\subset \Lc_1.
        \end{align*}
        Therefore,
        \begin{equation*}
            (a(1+D^*D)^{-1}a)^{\frac{p}{2}}-(a(1+DD^*)^{-1}a)^{\frac{p}{2}} \in \Lc_1.
        \end{equation*}
        
        We will now show that:       
        \begin{equation*}
            (1+D^*D)^{-p/2}a^p-(a(1+D^*D)^{-1}a)^{\frac{p}{2}} \in \Lc_1
        \end{equation*}
        and
        \begin{equation*}
            (1+DD^*)^{-p/2}a^p-(a(1+DD^*)^{-1}a)^{\frac{p}{2}} \in \Lc_1
        \end{equation*}
        by verifying the conditions of Theorem \ref{operator difference results} with $r = \frac{p}{2}>1$, $A = a^2$, and $B$ is either $(1+D^*D)^{-1}$
        or $(1+DD^*)^{-1}$. Take $B = (1+DD^*)^{-1}$. The case where $B = (1+D^*D)^{-1}$ will follow from an identical argument.

        First, to see Condition \eqref{op diff 1} in Theorem \ref{operator difference results} with $r = \frac{p}{2}$, we have:
        \begin{equation*}
            B^{\frac{p}{2}-1}A^{\frac{p}{2}-1} = (1+DD^*)^{-\frac{p-2}{2}}a^{p-2} 
        \end{equation*}
        Since $p > 2$, we have that $\frac{p}{p-2} > 1$ and so by the Araki-Lieb-Thirring inequality \eqref{ALT inequality},
        \begin{equation*}
            |(1+DD^*)^{-\frac{p-2}{2}}a^{p-2}|^{\frac{p}{p-2}} \prec\prec_{\log} (1+DD^*)^{-\frac{p}{2}}a^p.
        \end{equation*}
        However $(1+DD^*)^{-p/2}a^p$ is in $\Lc_{1,\infty}$, and therefore $(1+DD^*)^{-\frac{p-2}{2}}a^{p-2} \in \Lc_{\frac{p}{p-2},\infty}$,
        and this is what was required for Condition \eqref{op diff 1}.
                
        For Condition \eqref{op diff 2} in Theorem \ref{operator difference results}, we have: 
        \begin{equation*}
            A^{1/2}BA^{1/2} = a(1+DD^*)^{-1}a = |(1+DD^*)^{-1/2}a|^2. 
        \end{equation*} 
        Directly applying Lemma \ref{strong implies weak dimension} yields $(1+DD^*)^{-1/2}a \in \Lc_{p,\infty}$, and therefore $|(1+DD^*)^{-1/2}a|^2\in \Lc_{p/2,\infty}$,
        and this verifies Condition \eqref{op diff 2}.

        For Condition \eqref{op diff 3} in Theorem \ref{operator difference results}, we have:   
        \begin{align*}
            [BA^{1/2},A^{1/2}] &= [(1+DD^*)^{-1}a,a]\\
                               &= -(1+DD^*)^{-1}[DD^*,a](1+DD^*)^{-1}a\\
                               &= -(1+DD^*)^{-1}R_{|D^*|}(a)(1+DD^*)^{-1/2}a.
        \end{align*}
        By Lemma \ref{strong implies weak dimension}, we have that $(1+DD^*)^{-1/2}a \in \Lc_{p,\infty}$.
        By the Araki-Lieb-Thirring inequality \eqref{ALT inequality},
        \begin{equation*}
            |(1+DD^*)^{-1}R_{|D^*|}(a)|^{p/2} \prec\prec_{\log} (1+DD^*)^{-p/2}|R_{|D^*|}(a)|^{p/2}
        \end{equation*}
        We see that the right hand side is:
        \begin{equation*}
            (1+DD^*)^{-p/2}|R_{|D^*|}(a)|^{p/2} = (|R_{|D^*|}(a)|^{p/2}(1+DD^*)^{-p/2})^*
        \end{equation*}
        since $p/2 > 1$, we have that $|R_{|D^*|}(a)|^{p/2}(1+DD^*)^{-p/2} \in \Lc_{1,\infty}$ by the definition of being smoothly $p$-dimensional (Definition \ref{def_smoothly_p_dim}).
        Since the right hand side is in $\Lc_{1,\infty}$, it follows that $(1+DD^*)^{-1}R_{|D^*|}(a) \in \Lc_{p/2,\infty}$. 
        
        Thus, by the H\"older inequality,
        \begin{equation*}
            [BA^{1/2},A^{1/2}] \in \Lc_{p,\infty}\cdot\Lc_{p/2,\infty} \subseteq \Lc_{\frac{p}{3},\infty} \subset \Lc_{\frac{p}{2},1}
        \end{equation*}
        and so Condition \eqref{op diff 3} is verified.        
        
        Finally, for Condition \eqref{op diff 4} in Theorem \ref{operator difference results}, we have:
        \begin{align*}
            B^{\frac{p}{2}-1}[B,A^{\frac{p}{2}-1}]A &= (1+DD^*)^{-\frac{p-2}{2}}[(1+DD^*)^{-1},a^{p-2}]a^2\\
                                                    &= -(1+DD^*)^{-\frac{p}{2}}[DD^*,a^{p-2}](1+DD^*)^{-1}a^2\\
                                                    &= -(1+DD^*)^{-\frac{p}{2}}R_{|D^*|}(a^{p-2})(1+DD^*)^{-1/2}a^2.
        \end{align*}
        By the assumption of being smoothly $p$-dimensional, the first factor is in $\Lc_{1,\infty}$, and by Lemma \ref{strong implies weak dimension}, $(1+DD^*)^{-1/2}a^2 \in \Lc_{p,\infty}$.
        By the H\"older inequality, the product is in $\Lc_{\frac{p}{p+1},\infty}$, and so is trace class. This verifies Condition \eqref{op diff 4}.
        
        This verifies the four conditions of Theorem \ref{operator difference results}; and therefore:
        \begin{equation*}
            (1+D^*D)^{-p/2}a^p-(1+DD^*)^{-p/2}a^p \in \Lc_1
        \end{equation*}
        when $p > 2$, thus completing the proof.
    \end{proof}

\section*{Acknowledgements}
    The authors would like to extend our appreciation to Dr Vladimir Lotoreichik for many helpful discussions concerning the theory of self-adjoint extensions of the Laplace operator,
    and to Prof Matthias Lesch for his suggestion to use pseudodifferential operator theory to substantially simplify the present text.
    
    Edward McDonald was partly funded by an RTP scholarship. Fedor Sukochev greatly acknowledges the support of ARC grant FL170100052. Dmitriy Zanin was partly funded by a UNSW Scientia Fellowship.


\end{document}